\documentclass[10pt]{article}
\usepackage{graphicx} 

\usepackage{bigints}

\usepackage{amssymb}
\usepackage{mathrsfs}
\usepackage{amsmath}
\usepackage{amsfonts}

\renewcommand{\thefootnote}{\fnsymbol{footnote}}

\usepackage{amsthm}
\newtheorem{theo}{Theorem}[section]
\newtheorem{prop}{Proposition}[section]

\newtheorem{lemm}{Lemma}[section]
\newtheorem{fact}{Fact}[section]

\theoremstyle{definition}
\newtheorem{defi}{Definition}[section]

\theoremstyle{remark}

\newtheorem{rema}{Remark}[section]
\newtheorem{ques}{Question}[section]
\newtheorem{prob}{Problem}[section]

\usepackage{hyperref}

\title{Extremal metrics involving scalar curvature}
\author{Shota Hamanaka}
\date{\today}

\begin{document}

\maketitle

\renewcommand{\thefootnote}{\fnsymbol{footnote}} 
\renewcommand{\thefootnote}{\arabic{footnote}}

\renewcommand{\thefootnote}{\fnsymbol{footnote}} 
\renewcommand{\thefootnote}{\arabic{footnote}}
  
\begin{abstract}
We investigate extremal metrics at which various types of rigidity theorems involving scalar curvatures hold. 
The rigidity we discuss here is related to the rigidity theorems presented by Mario Listing in his previous preprint.
More specifically, we give some sufficient conditions for the metrics not to be rigid in this sense.
We also give several examples of Riemannian manifolds that satisfy such sufficient conditions.
\end{abstract}

\section{Introduction}\label{section-introduction}
~~~Llarull \cite{llarull1998sharp} showed some rigidity results for the standard sphere.
And Goette and Semmelmann \cite{goette2002scalar} generalized it to locally symmetric spaces of compact type and nontrivial Euler characteristic.
Later, Listing \cite{listing2010scalar} generalized their results in the following form.
\begin{theo}[{\cite[Theorem 1]{listing2010scalar}}]
\label{theo-listing-1}
    Let $(M_{0}^{n}, g_{0})~(n \ge 3)$ be an oriented spin closed Riemannian manifold with nonnegative curvature operator, positive Ricci curvature and non-vanishing Euler characteristic.
    Suppose that $(M^{n}, g)$ is an oriented closed Riemannian manifold and $f : M \rightarrow M_{0}$ is a spin map of non-zero degree.
    If the scalar curvature satisfies
    \[
    R_{g} \ge (R_{g_{0}} \circ f) \cdot \sqrt{\mathrm{area} (f)},
    \]
    then $\alpha := \mathrm{area} (f)$ is a (positive) constant and $f : (M, \alpha \cdot g) \rightarrow (M_{0}, g_{0})$ is a Riemannian covering.
    Here, $R_{g}, R_{g_{0}}$ denote the scalar curvature of $g, g_{0}$ respectively and 
    \[
    \mathrm{area}(f) : M \rightarrow [0, \infty) ;~x \mapsto \max_{v \in \Lambda^{2}T_{x}M \setminus \{0\}} \frac{f^{*}g_{0}(v,v)}{g(v,v)}.
    \]
\end{theo}
\noindent
For example, the standard even-dimensional spheres satisfy the assumption of Theorem \ref{theo-listing-1}.
For the case of $M = M_{0}$ and $f = \mathrm{id}_{M}$,
he also gave the following type of rigidity theorem.
\begin{theo}[{\cite[Theorem 2]{listing2010scalar}}]
\label{theo-listing-2}
    Suppose $(M^{n}, g_{0})$ is an oriented spin closed Riemannian manifold of dimension $n = 4k +1~(k \in \mathbb{N})$ with nonnegative curvature operator, positive Ricci curvature and non-vanishing Kervaire semicharacteristic $\sigma(M) \neq 0$.
    If $g$ is a Riemannian metric on $M$ satisfying
    \[
    R_{g} \ge R_{g_{0}} \cdot \| g_{0} \|_{2, g},
    \]
    then there is a positive constant $c > 0$ such that $g = c \cdot g_{0}$.
    Here, $\| g_{0} \|_{2, g} = \mathrm{area}(\mathrm{id}_{M})$ is defined by (\ref{eq-2-norm}) below.
\end{theo}
The condition that a metric has nonnegative curvature operator is preserved under the Ricci flow.
Moreover, on a closed manifold, the condition that a metric has positive Ricci curvature is also preserved under the Ricci flow for a sufficiently short time.
The Ricci flow solution $g(t)$ is homothetic (i.e. $g(t) = c(t) \phi(t)^{*} g_{0}$ where $c(t)$ is a positive constant and $\phi(t)$ is a diffeomorphism for each $t$) if and only if the initial metric $g_{0}$ is Einstein up to a diffeomorphism.
Hence, if $g_{0}$ is non-Einstein metric satisfying the assumption of the above theorem \ref{theo-listing-1} or \ref{theo-listing-2}, then it should be able to obtain a family of metrics (which is the solution of the Ricci flow equation starting at $g_{0}$) that satisfies the assumption of each theorem and is not merely a positive constant multiple of the original metric $g_{0}$.
However, at least in the author's opinion, metrics satisfying the aforementioned rigidity theorems must be special in some sense.
In light of these, we ask the following.
\begin{ques}
    Is there any non-Einstein metric $g_{0}$ that satisfies the above Listing's theorem \ref{theo-listing-1} or \ref{theo-listing-2}?
\end{ques}
On the other hand, the same quantity $R_{g} \cdot g$ appears in the context of Yamabe metrics.
\begin{theo}[\cite{shin1994examples}]\label{theo-kato}
    Let $M$ be a closed manifold and $g$ a Yamabe metric in its conformal class with $R_{g} > 0$.
    Assume that a metric $h$ on $M$ has a positive constant scalar curvature and satisfies
    \begin{equation}\label{eq-kato}
R_{g} \cdot g \ge R_{h} \cdot h~~\mathrm{on} ~M.
    \end{equation}
    Then, $h$ is also a Yamabe metric in its conformal class.
    Moreover, if the inequality in (\ref{eq-kato}) is strict, then $h$ is a unique Yamabe metric in its conformal class.
\end{theo}
\noindent
In particular, this gives rigidity in the space of positive csc metrics on $M$. More precisely, if $(M, h)$ is a positive csc metric on a closed manifold $M$ that is not a Yamabe metric, then no positive Yamabe metric $g$ satisfies (\ref{eq-kato}). 
\begin{rema}
    This type of sufficient condition is also known for other types of Yamabe metrics (\cite{hamanaka2021non, hamanaka2025notes}).
\end{rema}
And, for a conformal class $C$ on a closed manifold $M^n~(n \ge 2)$, its Yamabe constant $Y(M, C)$ is calculated by the infimum of the volume of the weighted measure ``$|R_g|\, d\mathrm{vol}_g$''.
\begin{fact}[{\cite[Lemma 1]{lebrun1999kodaira}}]
Let $C$ be a conformal class on a closed manifold $M^n$ of dimension $n \ge 2$. Then,
\[
    |Y(M, C)|^{n/2} = \inf_{g \in C} \int_M |R_g|^{n/2}\, d\mathrm{vol}_g.
    \]
\end{fact}
\noindent
Therefore, in particular, if $(M, g)$ is a Yamabe metric in its conformal class, then there is no metric $h \in [g]$ such that
\[
|R_g| \cdot g > |R_h| \cdot h~\mathrm{on}~M.
\]
\begin{rema}
Let $(M, g_{0})$ be a smooth Riemannian manifold and $\mathcal{M}$ the space of all Riemannian metrics on $M$.
Consider the functional 
\[
R_{min} : \mathcal{M} \ni g \mapsto \min_{M} R_{g} \in \mathbb{R}
\]
and a functional $\mu_{g_{0}}$ on $\mathcal{M}$, which is determined by $g_{0}$ and the scaling invariant of weight $-1$, i.e., $\mu(c \cdot g) = c^{-1} \mu(g)$ for all $c > 0$ and $g \in \mathcal{M}$.
If the metric $g_{0}$ is rigid with respect to the functional $\mu_{g_{0}}$ in a certain sense, then $\mu_{g_{0}}$ is an upper bound of $R_{min}$ as a functional on $\mathcal{M}$ and these values coincide at $g_{0}$.
On the other hand, when $M$ is closed $n$-manifold with non-positive Yamabe invariant $Y(M)$, the following functional is an upper bound of $R_{min}$ on $\mathcal{M}$:
\[
g \mapsto Y(M) \cdot \mathrm{Vol}(M, g)^{-2/n}.
\] 
Moreover, if a smooth metric $g_{0}$ attains equality, then it is a Yamabe metric (i.e., $Y(M, [g_{0}]) 
 = Y(M)$) and an Einstein metric.
However, when the Yamabe invariant is positive, this is not the case in general (see Remark \ref{rema-minscal} below or {\cite[Chapter 3]{listing2012scalar}}).
A Yamabe metric is expected to be standard in some sense (cf. {\cite[Section 1]{schoen1987variational}}), but it remains a difficult problem to know how to actually obtain it as a limit of the sequence of solutions to the Yamabe problem, and in what sense it is standard (see also Subsection \ref{subsection-singular} below).
\end{rema}
By the way, a certain infinitesimal feature of Einstein metrics relates to another type of scalar curvature rigidity.
\begin{theo}[{\cite[Theorem 1.5, Remark 1.6]{kroencke2024stability}}]
    Let $(M, g)$ be a closed Einstein manifold.
    If its is semi-stable and integrable, then it is locally scalar curvature rigid, i.e., there is no metric $\hat{g}$ $C^{\infty}$-close to $g$ such that 
    $g - \hat{g} |_{M \setminus K} \equiv 0$, $\mathrm{Vol}(K, g) = \mathrm{Vol}(M, \hat{g})$ 
    for some compact set $K \subset M$, and 
    $R_{\hat{g}} \ge R_{g}$ on $M$ and $R_{\hat{g}}(x) > R_{g}(x)$ for some $x \in M$.
    If $(M, g)$ is unstable, it is not locally scalar curvature rigid.
\end{theo}
\begin{rema}
    Dahl--Kr\"{o}ncke \cite{dahl2024local, kroencke2024stability} also found a relation  between stability of Einstein metrics and certain type of scalar curvature rigidity.
\end{rema}

As noted in \cite{goette2002scalar}, one can apply the construction of Lohkamp \cite{lohkamp1999scalar} to see that not all metrics on $M$ are area-extremal (for this definition, see the beginning of Section \ref{section-rigidity}) if $\mathrm{dim} M \ge 3$.
Similarly, for each closed manifold $M^{n}$ of dimension $n \ge 3$, there is at least one Riemannian metric that is not type III scalar curvature rigid in the sense of Listing (see Definition \ref{def-3} below).
However, such an example is not given in an explicit way.
That is, we can deduce the existence of such a metric but we cannot know any concrete properties involving its curvatures in general.
In light of the above, this paper aims to investigate relations between Listing-type of rigidity phenomena involving scalar curvature and standard metrics in various senses.
In particular, we will investigate the relations between such types of scalar curvature rigidity and some infinitesimal features for constant scalar curvature metrics.

Throughout the paper, any Riemannian metric will be smooth.
The Ricci curvature and the scalar curvature of a metric $g$ are denoted by $\mathrm{Ric}_{g}$ and $R_{g}$, respectively.
We often abbreviate ``constant scalar curvature'' as ``csc''.
And, $ \overset{\circ}{\mathrm{Ric}_{g}} := \mathrm{Ric}_{g} - (R_g / n) g$ denotes the trace-less Ricci tensor of $g$.
The (non-positive) Laplacian that acts on functions is defined by $\Delta_{g} f = \mathrm{tr}_{g} \nabla^{g} df$, where $\nabla^{g}$ is the Levi-Civita connection of $g$.
The volume element is denoted by $\mathrm{vol}_{g}$.
For two symmetric $(0, 2)$-tensors $g$ and $h$, we say $g \ge h$ on $\Lambda^{2}TM$ if $g (v, w) \ge h(v, w)$ for all $v, w \in \Lambda^{2}TM$.

It is well-known that for a fixed conformal class $C$ on a closed manifold $M^n~(n \ge 3)$, $g \in C$ is a critical point of the normalized Einstein--Hilbert functional 
\[
g \mapsto E(g) := \frac{\int_{M} R_{g}\, d\mathrm{vol}_{g}}{ \mathrm{Vol}(M,g)^{\frac{n-2}{n}}}
\]
if and only if its scalar curvature $R_g$ is constant.
When a metric $g$ on a closed manifold $M^n~(n \ge 3)$ has a constant scalar curvature $R_g$ and $\{ g_t \} \subset [g]$ is a smooth deformation of $g$ in its conformal class, the second variation of the normalized Einstein--Hilbert functional along $\{ g_t \}$ is given by (see \cite{kobayashi2013yamabe})
\[
\left. \frac{d^2}{dt^2} E(g_t) \right|_{t=0} = \frac{(n-1)(n-2)}{2n^2} \mathrm{Vol}(M,g)^{\frac{2-n}{n}} \int_{M} \varphi \left( -\Delta_{g} - \frac{R}{n-1} \right) \varphi\, d\mathrm{vol}_{g},
\]
where
\[
\varphi = \mathrm{tr}_{g} \left( \left. \frac{d}{dt} g_t \right|_{t=0} \right) - \frac{\int_{M} \mathrm{tr}_{g} \left( \left. \frac{d}{dt} g_t \right|_{t=0} \right)\, d\mathrm{vol}_{g}}{\int_{M} d\mathrm{vol}_{g}}.
\]
Therefore, when $n \ge 3$, $g \in C$ is a stable critical point of the normalized Einstein--Hilbert functional if and only if $R_g = \mathrm{const}$ and 
\[
\lambda_1 (-\Delta_g) \ge \frac{R_g}{n-1}.
\]
Note that every metric with nonpositive constant scalar curvature is a stable critical point of the normalized Einstein--Hilbert functional restricted on a fixed conformal class since $\lambda_1 (-\Delta_g) > 0$ for all metrics $g$ on a closed manifold.
Our first main result suggests a relation between unstabilty of csc metrics and Listing-types of rigidity.
\begin{theo}\label{theo-main-unstable}
    Let $M^n$ be a closed manifold of dimension $n \ge 3$.
    Assume that $g$ is a csc metric with $\frac{R_g}{n-1} \neq \lambda_1(-\Delta_g)$, where $\lambda_1(-\Delta_g)$ denotes the first eigenvalue of the Laplace operator $-\Delta_g$ acting on functions on $M$.
    Let $v$ be the first eigenfunction of $-\Delta_g$ and $N(v)$ be its nodal set of $v$, i.e. 
    \[
    N(v) := \{ x \in M \mid v(x) = 0 \}.
    \]
    Then, there is a function $u \in C^{\infty}(M \setminus N(v))$ and sufficiently small $s_0 >0$ such that
    $g_s := g -s \left( \nabla_g^2 u - \frac{\Delta_g u}{n}g \right)$ is a smooth Riemannian meric on $M \setminus N(v)$ and
    \[
    R_{g_s} > R_g \cdot \| g \|^2_{1, g_s}~\mathrm{and}~R_{g_s} > R_g \cdot \| g \|_{2, g_s}~~\mathrm{on}~M \setminus N(v)
    \]
    for all $s \in (0, s_{0}]$.
    In particular, if $g$ is unstable in the above sense, then the assertion holds.
\end{theo}
\noindent
From \cite{cheng1976eigenfunctions, hardt1989nodal}, the nodal set $N(v)$ decomposes into the disjoint union $\{ x \in N(v) \mid |\nabla v|(x) > 0 \} \sqcup \{ x \in N(v) \mid |\nabla v|(x) = 0 \}$ of smooth $(n-1)$-manifold and a closed countably $(n-2)$-rectifiable subset.
Moreover, from \cite{naber2017volume}, there is a constant $C = C(g, \lambda_1 (-\Delta_g)) > 0$ such that
\[
\mathrm{Vol}(B_r(N(v))) \le C r,
\]
where $B_r(N(v))$ denotes the $r$-neighborhood of the nodal set $N(v)$.
Thus, if we can prove a Listing-type of rigidity theorem for a metric with this type of singularities, it will provide a new sufficient condition for the metric to be a stable csc metric.
However, from the observation of Section \ref{section-proof}, every metric on a closed manifold can be deformed in the conformal direction so that the Listing-type quantity of the deformation strictly increases outside a compact set ($=$ the nodal set of the first eigenfunction of the metric).
Moreover, this deformation can be taken so that each deformed metric is Lipschitz continuous.
Hence, when considering the singular version of Listing-type rigidity theorems as mentioned above, it is necessary to restrict the consideration to deformations that are transverse to the conformal direction.
By the way, several generalizations of the Llarull-types of rigidity theorem are known for several types of singular metrics \cite{chu2024llarull, lee2022rigidity}.

However, in the smooth category, the fact that a metric $g$ is Type II scalar curvature rigid in the sense of Definition \ref{def-2} holds does not necessarily imply that $\lambda_{1}(-\Delta_g) = \frac{R_g}{n-1}$ holds.
Indeed, the Fubini-Study metric $g_{FS}$ on $\mathbb{C}P^{n}~(n \ge 3)$ satisfies the assumption of Listing's rigidity theorem {\cite[Theorem 1]{listing2010scalar}} and hence it is Type II scalar curvature rigid in the sense of Definition \ref{def-2} below.
However, it satisfies $\frac{R_{g_{FS}}}{2n-1} = \frac{4n(n+1)}{2n-1} < \lambda_1 (-\Delta_{g_{FS}}) = 4(n+1)$ since $n \ge 3$.
On the other hand, the Fubini-Study metric on $\mathbb{C}P^1$ is Type I scalar curvature rigid in the sense of Definition \ref{def-1} and satisfies $\lambda_{1}(-\Delta_{g_{FS}}) = R_{g_{FS}}$.
\begin{ques}
    Does any metric $g$ on a closed $n$-manifold $M^n$ that is Type I scalar curvature rigid in the sense of Definition \ref{def-1} satisfy $\lambda_{1}(-\Delta_{g}) = \frac{R_{g}}{n-1}$?
\end{ques}

We will show below that a metric which is not Einstein, and furthermore satisfies a certain assumption, is not a Listing-type extremal metric in a certain sense.
\begin{theo}\label{theo-1}
Let $(M^{n}, g)$ be a closed Riemannian manifold of dimension $n \ge 2$ and a smooth function $f \in C^{\infty}(M)$ satisfying 
     \begin{equation}\label{eq-assumption-1}
     \begin{split}
         & \left\langle \mathrm{Hess} f + \nabla f \otimes \nabla f,  \overset{\circ}{\mathrm{Ric}_{g}} \right\rangle_{g} - \| \overset{\circ}{\mathrm{Ric}_{g}} \|_{g}^{2}   \\
        &\qquad+ \left( 1 - \frac{1}{n} \right) \langle \nabla f, \nabla R_{g} \rangle _{g} + \left( \frac{1}{2} - \frac{1}{n} \right) \Delta_{g} R_{g} - R_{g} \cdot \max_{v \in T_{x}M,\, |v|_{g} = 1 }\overset{\circ}{\mathrm{Ric}_{g}}
     \end{split}
     \end{equation}
     is positive (resp. negative) on $M$.
    Then, there is a small constant $s_{0} > 0$ depending only on $n, M$ and $g$ such that for all $0 < s \le s_{0}$, $g_{s} = g +s\, \overset{\circ}{\mathrm{Ric}_{g}}$ (resp. $g - s\, \overset{\circ}{\mathrm{Ric}_{g}}$) is a Riemmanian metric on $M$ and that 
    \begin{equation}\label{eq-strong-rigid}
    R_{g_{s}} > R_{g} \cdot \| g \|^{2}_{1, g_{s}}
    \end{equation}
    on $M$.
    Here, $\| g \|_{1, g_{s}} : M \rightarrow [0, \infty)$ is the function on $M$ defined by
    \begin{equation}\label{eq-1-norm}
    \| g \|_{1, g_{s}}(x) := \sqrt{\max_{v \in T_{x}M \setminus \{ 0 \}} \frac{g(v, v)}{g_{s}(v, v)}}.
    \end{equation}
    In particular, $g$ is not type I scalar curvature rigid in the sense of Definitions \ref{def-1}.
\end{theo}
\begin{rema}
    In particular, if $(M, g)$ has a positive constant scalar curvature and satisfies
    \begin{equation}\label{eq-uniform}
    \| \overset{\circ}{\mathrm{Ric}_{g}} \|_{g}(x) \neq 0~~\mathrm{for~all}~x \in M,
    \end{equation}
    then $(M, g)$ satisfies the assumption in the first line of (\ref{eq-assumption-1}).
\end{rema}
Similarly to Theorem \ref{theo-1}, we can also prove the following.
\begin{theo}\label{theo-2}
Let $(M^{n}, g)$ be a closed Riemannian manifold of dimension $n \ge 2$ and a smooth function $f \in C^{\infty}(M)$ satisfying 
     \begin{equation*}
     \begin{split}
         & \left\langle \mathrm{Hess} f + \nabla f \otimes \nabla f,  \overset{\circ}{\mathrm{Ric}_{g}} \right\rangle_{g} - \| \overset{\circ}{\mathrm{Ric}_{g}} \|_{g}^{2}   \\
        &\qquad+ \left( 1 - \frac{1}{n} \right) \langle \nabla f, \nabla R_{g} \rangle _{g} + \left( \frac{1}{2} - \frac{1}{n} \right) \Delta_{g} R_{g} - \frac{R_{g}}{2} \cdot \max_{v \in \Lambda^{2}T_{x}M \setminus \{ 0 \} } g^{-1} \overset{\circ}{\mathrm{Ric}_{g}}(v,v),
     \end{split}
     \end{equation*}
     where for $(x, v) = (x, v_{1} \wedge v_{2}) \in \Lambda^{2} TM$,
     \[
     \begin{split}
     g^{-1} h(v,v) &= \mathrm{tr} \left(
     \begin{pmatrix}
         g(v_{1}, v_{1}) & g(v_{1}, v_{2}) \\
         g(v_{2}, v_{1}) & g(v_{2}, v_{2})
     \end{pmatrix}^{-1} \cdot
     \begin{pmatrix}
         h(v_{1}, v_{1}) & h(v_{1}, v_{2}) \\
         h(v_{2}, v_{1}) & h(v_{2}, v_{2})
     \end{pmatrix} \right) \\
     &= \frac{1}{g(v,v)} \left( g(v_{2}, v_{2}) h(v_{1}, v_{1}) - 2g(v_{1}, v_{2}) h(v_{1}, v_{2}) + g(v_{1}, v_{1}) h(v_{2}, v_{2}) \right),
     \end{split}
     \]
     is positive (resp. negative) on $M$.
    Then, there is a small constant $s_{0} > 0$ depending only on $n, M$ and $g$ such that for all $0 < s \le s_{0}$, $g_{s} = g +s\, \overset{\circ}{\mathrm{Ric}_{g}}$ (resp. $g - s\, \overset{\circ}{\mathrm{Ric}_{g}}$) is a Riemmanian metric on $M$ and that 
     \begin{equation}\label{eq-rigid-2}
     R_{g_{s}} > R_{g} \cdot \| g \|_{2, g_{s}}
    \end{equation}
    on $M$.
    Here, $\| g \|_{2, g_{s}} : M \rightarrow [0, \infty)$ is the function on $M$ defined by
    \begin{equation}\label{eq-2-norm}
    \| g \|_{2, g_{s}}(x) := \sqrt{\max_{v \in \Lambda^{2} T_{x}M \setminus \{ 0 \}} \frac{g(v, v)}{g_{s}(v, v)}}.
    \end{equation}
    In particular, $g$ is not type II scalar curvature rigid in the sense of Definitions \ref{def-2}.
\end{theo}
\begin{theo}\label{theo-3}
     Let $(M^{n}, g)$ be a closed Riemannian manifold of dimension $n \ge 2$ and a smooth function $f \in C^{\infty}(M)$ satisfying 
     \begin{equation*}
     \begin{split}
         &\left\{ \left\langle \mathrm{Hess} f + \nabla f \otimes \nabla f,  \overset{\circ}{\mathrm{Ric}_{g}} \right\rangle_{g} - \| \overset{\circ}{\mathrm{Ric}_{g}} \|_{g}^{2} \right.  \\
        &\qquad\qquad\left. + \left( 1 - \frac{1}{n} \right) \langle \nabla f, \nabla R_{g} \rangle _{g} + \left( \frac{1}{2} - \frac{1}{n} \right) \Delta_{g} R_{g} \right\} g + R_{g} \cdot \overset{\circ}{\mathrm{Ric}_{g}}
     \end{split}
     \end{equation*}
     is positive (resp. negative) definite on $TM$.
    Then, there is a small constant $s_{0} > 0$ depending only on $n, M$ and $g$ such that for all $0 < s \le s_{0}$, $g_{s} = g +s\, \overset{\circ}{\mathrm{Ric}_{g}}$ (resp. $g - s\, \overset{\circ}{\mathrm{Ric}_{g}}$) is a Riemmanian metric on $M$ and that 
    \begin{equation}\label{eq-strong-rigid}
     R_{g_{s}} \cdot g_{s} > R_{g} \cdot g
    \end{equation}
    on $TM$.
    In particular, $g$ is not type III scalar curvature rigid in the sense of Definitions \ref{def-3}.
\end{theo}
\begin{rema}
    In particular, if a metric $g$ satisfies $R_g$ is a positive constant,  $\mathrm{Ric}_{g} > 0$ on $M$ and there is a function $u \in C^{\infty}_{+}(M)$ such that $\left\langle \mathrm{Hess}\, u, \mathrm{Ric}_{g} - \frac{R_{g}}{n} g \right\rangle - \| \mathrm{Ric}_{g} \|_{g}^{2}\, u > 0$ on $M$, then $g$ satisfies the above assumption.
\end{rema}
\begin{theo}\label{theo-4}
    Let $(M^{n}, g)$ be a closed Riemannian manifold of dimension $n \ge 2$ and a smooth function $f \in C^{\infty}(M)$ satisfying 
     \begin{equation*}
     \begin{split}
         & R_{g} \left\{\left\langle \mathrm{Hess} f + \nabla f \otimes \nabla f,  \overset{\circ}{\mathrm{Ric}_{g}} \right\rangle_{g} - \| \overset{\circ}{\mathrm{Ric}_{g}} \|_{g}^{2} \right.  \\
        &\qquad\left.+ \left( 1 - \frac{1}{n} \right) \langle \nabla f, \nabla R_{g} \rangle _{g} + \left( \frac{1}{2} - \frac{1}{n} \right) \Delta_{g} R_{g} \right\} g + \frac{R_{g}^{2}}{2} \cdot g^{-1} \overset{\circ}{\mathrm{Ric}_{g}}
     \end{split}
     \end{equation*}
     is positive (resp. negative) definite on $\Lambda^{2} TM$.
    Then, there is a small constant $s_{0} > 0$ depending only on $n, M$ and $g$ such that for all $0 < s \le s_{0}$, $g_{s} = g +s\, \overset{\circ}{\mathrm{Ric}_{g}}$ (resp. $g_{s} = g - s\, \overset{\circ}{\mathrm{Ric}_{g}}$) is a Riemmanian metric on $M$ and that 
    \begin{equation}\label{eq-strong-rigid-2}
     R^{2}_{g_{s}} \cdot g_{s} > R^{2}_{g} \cdot g
    \end{equation}
   on $\Lambda^{2} TM$.
   In particular, $g$ is not type IV scalar curvature rigid in the sense of Definitions \ref{def-4}.
\end{theo}
\begin{rema}\label{rema-assumption}
\begin{itemize}
    \item All the assumptions of the above Theorems \ref{theo-1}, \ref{theo-2}, \ref{theo-3} and \ref{theo-4} especially imply that $R_{g}$ is not a constant on $M$ or
    \[
    \| \overset{\circ}{\mathrm{Ric}_{g}} \|_{g} \neq 0~~\mathrm{on}~M.
    \]
    \item Since every Einstein metric does not satisfy any of the above assumptions (see Remark \ref{rema-assumption} above), our theorems above cannot be applied to Einstein metrics.
    \end{itemize}
\end{rema}

Finally, we can also deform a non-Einstein unique positive Yamabe metric in its conformal class towards csc metrics for which Listing-types of rigidity do not hold provided that the traceless Ricci tensor of the metric satisfies a certain assumption.
\begin{theo}\label{theo-cscdeformation}
    Let $(M^n, g)~(n \ge 3)$ be a closed Riemannian manifold.
    Assume that $g$ is a unit volume unique nonnegative Yamabe metric in its conformal class $[g]$, which is not conformally equivalent to the standard sphere.
    Assume that $g$ is not an Einstein metric.
    When $R_g > 0$, assume moreover that
    \begin{equation}\label{eq-tracelessricci-assumption}
    \min_{v \in T_x M,~|v|_g = 1} \mathrm{pr}_{T_{g}\mathcal{C}_{1}} \overset{\circ}{\mathrm{Ric}}_g(v,v) \le 0~~\mathrm{on}~M.
    \end{equation}
    \begin{equation}\label{eq-tracelessricci-assumption2}
   \left( \mathrm{resp.}~ \min_{v \in \Lambda^2 T_x M,~|v|_g = 1} \mathrm{pr}_{T_{g}\mathcal{C}_{1}} \overset{\circ}{\mathrm{Ric}}_g(v,v) \le 0~~\mathrm{on}~M. \right)
    \end{equation}
    Then, there is a sufficiently small constant $s_0 > 0$ and a smooth deformation of csc metrics $\{ \gamma_t \}_{t \in [0, s_0]} \subset \mathcal{C}_1$ such that on $M$,
    \[
    R_{\gamma_t} > R_g \cdot \| g \|_{1, \gamma_t}^2~~\left( \mathrm{resp.}~~R_{\gamma_t} > R_{g} \cdot \| g \|_{2, \gamma_t} \right)
    \]
    for all $s \in (0, s_{0}]$.
    Here, $\mathcal{C}_1$ denotes the space of all csc metrics with unit volume and $\mathrm{pr}_{T_g \mathcal{C}_1}$ denotes the projection onto the tangent space $T_g \mathcal{C}_1$ of $\mathcal{C}_1$ at $g$.
    In particular, $g$ is neither type I nor type II scalar curvature rigid in the sense of Definitions \ref{def-1} and \ref{def-2}. 
\end{theo}
\begin{rema}
\begin{itemize}
\item[(1)] If $\mathrm{tr}_{g} \left( \mathrm{pr}_{T_g \mathcal{C}_1} (\overset{\circ}{\mathrm{Ric}_g}) \right) = 0$ on $M$, then the assumption (\ref{eq-tracelessricci-assumption}) is satisfied. 
In particular, if $\mathrm{pr}_{T_{g} \mathcal{C}_1} (\overset{\circ}{\mathrm{Ric}_g}) = \overset{\circ}{\mathrm{Ric}_g}$, then (\ref{eq-tracelessricci-assumption}) is satisfied.
   \item[(2)] To be precise, the tangent space $T_{g} \mathcal{C}_1$ is 
    \[
T_{g} \mathcal{C}_1 = \mathrm{Ker}\, \alpha_g \cap \left\{ h \in T_{g}\mathcal{M} \left|~\int_M \mathrm{tr}_g h\, d\mathrm{vol}_g = 0 \right. \right\},
    \]
    where $T_{g}\mathcal{M}$ is the tangent space of the space of all Riemannian metrics $\mathcal{M}$ at $g$ and
    \[
    \alpha_g (h) = \Delta_g (\Delta_g \mathrm{tr}_g h + \delta_g (\delta_g h) - \langle \mathrm{Ric}_g, h \rangle).
    \]
    Since $\overset{\circ}{\mathrm{Ric}_g}$ is traceless and $R_g$ is constant, for a csc metric $g$ with unit volume, $\mathrm{pr}_{T_{g} \mathcal{C}_1} (\overset{\circ}{\mathrm{Ric}_g}) = \overset{\circ}{\mathrm{Ric}_g}$ if and only if $\Delta_{g} \| \overset{\circ}{\mathrm{Ric}_g} \|^2 = 0$ on $M$.
    In particular, if $M$ is closed connected, it is equivalent to that $\| \overset{\circ}{\mathrm{Ric}_g} \| = \mathrm{const}$ on $M$.
    Several examples satisfying this condition are given in \ref{subsection-lie}, \ref{subsection-submersions} and \ref{subsection-unstable} below.
    \item[(3)] For each sufficiently small $t$, $\gamma_t$ constructed above is also the unique csc metric in its conformal class (up to rescaling).
    From Kato's theorem (Theorem \ref{theo-kato}), the unique-csc assumption of this theorem is natural. 
    \end{itemize}
\end{rema}

This paper is organized as follows.
In Section \ref{section-rigidity}, we give several types of definition of scalar curvature rigidity and review some results for certain Riemannian metrics with positive scalar curvature and a relation between ``scalar minimum functional'' and an extremal metric (see Remark \ref{rema-minscal}).
In Section \ref{section-preliminaries}, we describe a formula that is necessary to prove our main theorems.
Furthermore, we consider statements of the same type as our main theorems on compact manifolds with boundary and non-compact complete manifolds.
In Section \ref{section-proof}, we prove our main theorems.
In Section \ref{section-examples}, through several examples, we examine various metrics that are not scalar curvature-rigid in our sense.
In Section \ref{section-conclusion}, we present some further questions related to our main theorems.
In Section \ref{section-appendix} (Appendix), we give a proof of the formula in Section \ref{section-preliminaries}.

\section{Previous rigidity results for metrics with positive scalar curvature}\label{section-rigidity}
A metric $g$ on a smooth manifold $M$ is called \textit{(globally) area-extremal} if, for a metric $h$ satisfying $h \ge g$ on $\Lambda^{2}TM$, $R_{h} \ge R_{g}$ holds only when $R_{h} = R_{g}$ on $M$.
As a generalization of Llarull's significant rigidity result \cite{llarull1998sharp}, Goette and Semmelmann \cite{goette2002scalar} gave a sufficient condition for a metric to be locally area-extremal as follows.
\begin{prop}[{\cite[Lemma 0.3]{goette2002scalar}}]
   Let $(M, g)$ be a compact Riemannian manifold whose Ricci curvature $\mathrm{Ric}_{g}$ is positive definite on $M$.
   Then there exists no nonconstant $C^{1}$-path $(g_{t})_{t \in [0, \varepsilon]}$ of Riemannian metrics on $M$ for $\varepsilon > 0$ with $g_{0} = g$, such that $g_{t} \ge g$ on $TM$ and $R_{g_{t}} \ge R_{g_{0}}$ on $M$.

   Suppose moreover that $2\, \mathrm{Ric}_{g} - R_{g} \cdot g$ is negative definite on $M$.
   Then there is no nonconstant path $(g_{t})_{t \in [0, \varepsilon]}$ as above, such that $g_{t} \ge g$ on $\Lambda^{2}TM$ and $R_{g_{t}} \ge R_{g}$ on $M$.
\end{prop}
Meanwhile, they also gave the following stability result.
\begin{theo}[{\cite[Theorem 2.4]{goette2002scalar}}]
     Let $(M_{0}^{n}, g_{0})~(n \ge 3)$ be an oriented closed Riemannian manifold with nonnegative curvature operator, positive Ricci curvature and non-vanishing Euler characteristic.
    Suppose that $(M^{n}, g)$ is an oriented closed Riemannian manifold and $f : M \rightarrow M_{0}$ is a spin map of non-vanishing $\hat{A}$-degree $\mathrm{deg}_{\hat{A}}(f) \neq 0$ and $\mathrm{area}(f) \le 1$.
    Then $R_{g} \ge R_{g_{0}} \circ f$ implies that $R_{g} = R_{g_{0}} \circ f$.
    If moreover, $\mathrm{Ric}_{g} > 0$ and $2\, \mathrm{Ric}_{g} - R_{g} \cdot g < 0$ on $M$, then $f : M \rightarrow M_{0}$ is a Riemannian submersion.
\end{theo}
They also prove area-extremality and rigidity for a certain class of metrics with nonnegative curvature operator on $\Lambda^{2}TM$.
\begin{theo}[\cite{goette2002scalar}]
    Let $(M, g)$ be a compact connected oriented Riemanniam manifold with nonnegatuve curvature operator on $\Lambda^{2}TM$, such that the universal covering of $M$ is homeomorphic to a symmetric space $G/K$ of compact type with $\mathrm{rk}G \le \mathrm{rk}K + 1$.
    Then $g$ is (globally) area-extremal. 
    If moreover, $\mathrm{Ric}_{g} > 0$ and $2\, \mathrm{Ric}_{g} - R_{g} \cdot g < 0$ on $M$, then $R_{h} \ge R_{g}$ and $h \ge g$ on $\Lambda^{2}TM$ implies $h = g$.
\end{theo}
\noindent
Later Listing generalized these to Theorem \ref{theo-listing-1} and \ref{theo-listing-2} above.
On the other hand, Lott \cite{lott2021index} extended results of Llarull and Goette--Semmelmann to manifolds with boundary.

According to Listing's work \cite{listing2010scalar}, we define four types of rigidity of metrics involving scalar curvature.
\begin{defi}\label{def-1}
    Let $M$ be a smooth manifold.
    A metric $g_{0}$ on $M$ is \textit{type I scalar curvature rigid in the sense of Listing} if
    for any metric $g$ on $M$,
    \begin{equation}\label{eq-def-1}
    R_{g} \ge R_{g_{0}} \cdot \| g_{0} \|^{2}_{1, g}
    \end{equation}
    implies that $g = c \cdot g_{0}$ for some positive constant $c > 0$.
    Here, $\| g \|_{1, g_{s}}$ is the function defined in (\ref{eq-1-norm}).
\end{defi}

\begin{defi}\label{def-2}
    Let $M$ be a smooth manifold.
    A metric $g_{0}$ on $M$ is \textit{type II scalar curvature rigid in the sense of Listing} if
    for any metric $g$ on $M$,
    \begin{equation}\label{eq-def-2}
    R_{g} \ge R_{g_{0}} \cdot \| g_{0} \|_{2, g},
    \end{equation}
    implies that $g = c \cdot g_{0}$ for some positive constant $c > 0$.
    Here, $\| g \|_{2, g_{s}}$ is the function defined in (\ref{eq-2-norm}).
\end{defi}

\begin{defi}\label{def-3}
    Let $M$ be a smooth manifold.
    A metric $g_{0}$ on $M$ is \textit{type III scalar curvature rigid in the sense of Listing} if
    for any metric $g$ on $M$, 
    \begin{equation}\label{eq-def-3}
    R_{g} \cdot g \ge R_{g_{0}} \cdot g_{0}~~\mathrm{on}~TM,
    \end{equation}
    implies that $g = c \cdot g_{0}$ for some positive constant $c > 0$.
\end{defi}
\begin{defi}\label{def-4}
    Let $M$ be a smooth manifold.
    A metric $g_{0}$ on $M$ is \textit{type IV scalar curvature rigid in the sense of Listing} if
    for any metric $g$ on $M$,
    \begin{equation}\label{eq-def-4}
    R_{g}^{2} \cdot g \ge R_{g_{0}}^{2} \cdot g_{0}~~\mathrm{on}~\Lambda^{2}TM,
    \end{equation}
    implies that $g = c \cdot g_{0}$ for some positive constant $c > 0$.
\end{defi}
\begin{rema}
    The condition (\ref{eq-def-3}) implies (\ref{eq-def-1}).
    And, if $R_{g_{0}} \ge 0$ on $M$, then the condition (\ref{eq-def-2}) is equivalent to (\ref{eq-def-4}).
\end{rema}

\begin{rema}\label{rema-minscal}
Here, we summarize some of the fundamental facts about the scalar minimum functional (denoted by $R_{min}$ below) and scalar curvature rigid metrics.

Let $M^{n}$ be a closed manifold of dimension $n \ge 2$ and $\mathcal{M}(M)$ the space of all Riemannian metrics on $M$.
    Consider the following \textit{scalar minimum functional}:
    \[
    R_{min} : \mathcal{M}(M) \rightarrow \mathbb{R}~;~g \mapsto \min_{M} R_{g}.
    \]
    For a fixed Riemannian metric $g_{0} \in \mathcal{M}(M)$, we define the following two functionals $F_{1, g_{0}}$ and $F_{2, g_{0}}$.
        \[
        \begin{split}
        F_{1, g_{0}} &: \mathcal{M}(M) \rightarrow \mathbb{R}~;~ g \mapsto \max_{M} R_{g_{0}} \cdot \| g_{0} \|^{2}_{1, g}, \\
        F_{2, g_{0}} &: \mathcal{M}(M) \rightarrow \mathbb{R}~;~ g \mapsto \max_{M} R_{g_{0}} \cdot \| g_{0} \|_{2, g}.
        \end{split}
        \]
    Then, from the definitions of scalar curvature rigid metrics of  types I and II, if $g_{0}$ is a type I (resp. type II) scalar curvature rigid in the sense of Listing, then $F_{1, g_{0}}$ (resp. $F_{2, g_{0}}$) is an upper bound of $R_{min}$ as a functional on $\mathcal{M}(M)$.
    That is, it holds that $F_{1, g_{0}}(g) \ge R_{min}(g)$ (resp. $F_{2, g_{0}}(g) \ge R_{min}(g)$) for all $g \in \mathcal{M}(M)$.
    \begin{proof}
        Suppose there is a metric $g \in \mathcal{M}(M)$ such that $R_{min}(g) =\min_{M} R_{g} > F_{1, g_{0}}(g)$.
        Then, from Definition \ref{def-1}, there is a positive constant $c > 0$ such that $g = c \cdot g_{0}$.
        Hence, $R_{g} = c^{-1} R_{g_{0}} \le c^{-1} F_{1, g_{0}}(g_{0}) = F_{1, g_{0}}(g)$ on $M$.
        This contradicts our supposition $R_{min}(g) =\min_{M} R_{g} > F_{1, g_{0}}(g)$.
        The proof for the corresponding statement to $F_{2, g_{0}}$ is similar.
    \end{proof}
    \noindent
    Moreover, each equality is attained by the scalar curvature rigid metric $g_{0}$ if $R_{g_{0}}$ is constant on $M$.
    On the other hand, Gromov \cite{gromov1996positive} introduced the \textit{$K$-area} of $M$ and gave an upper bound of $R_{min}$ on closed spin $n$-manifolds (``\textit{$K$-area inequality}'' in {\cite[$5 \frac{1}{4}$]{gromov1996positive}}), which is expressed using the $K$-area and the dimension $n$. See also \cite{listing2012scalar} for more detail.
    
    The functional $R_{min}$ also appears in certain characterizations of the Yamabe constant and the Yamabe invariant when it is nonpositive. As pointed out in {\cite[Section 3]{listing2012scalar}},
    when the Yamabe invariant, alias the sigma constant, $Y(M^{n})$ is non-positive, then
    \[
    \mathcal{M} \ni g \mapsto \mu(g) := Y(M^{n}) \cdot \mathrm{Vol}(M, g)^{-2/n} \in \mathbb{R}
    \]
    is an upper bound of $R_{min}$ as a functional on $\mathcal{M}(M)$.
    Indeed, for a conformal class $C$ on $M$, if its Yamabe constant $Y(M ,C)$ is non-positive, then 
    \[
    Y(M, C) = \sup_{g \in C} R_{min}(g) \cdot \mathrm{Vol}(M, g)^{2/n}
    \]
    (see {\cite[Corollary 5.16]{kobayashi2013yamabe}}).
    Hence, 
    \[
    Y(M) = \sup_{g \in \mathcal{M}(M)} R_{min}(g) \cdot \mathrm{Vol}(M, g)^{2/n}
    \]
    if $Y(M) \le 0$. Moreover, if a smooth metric $g$ attains the equality, then it is a Yamabe metric (i.e., $Y(M, [g]) 
 = Y(M)$) and an Einstein metric (see \cite{obata1971conjectures}).
 Here, the Yamabe invariant $Y(M^{n})$ is defined as follows.
 \[
 Y(M^{n}) := \sup_{C} Y(M, C)
 := \sup_{C} \inf_{h \in C} \frac{\int_{M} R_{h}\, d\mathrm{vol}_{h}}{\mathrm{Vol}(M, h)^{\frac{n-2}{n}}},
 \]
where the supremum is taken over all conformal classes on $M$.
\end{rema}

\section{Preliminaries}\label{section-preliminaries}
The following first variation formula of scalar curvature functional is well-known (see \cite{besse2007einstein}).
\begin{lemm}\label{lemm-scal-variation}
    $DR|_{\bar{g}} (h) = -\Delta_{\bar{g}} (tr_{\bar{g}} h) + \mathrm{div}_{\bar{g}} (\mathrm{div}_{\bar{g}} h) - \langle \mathrm{Ric}_{\bar{g}}, h \rangle_{\bar{g}}.$
    Here, $\Delta_{\bar{g}} f = \mathrm{tr}_{\bar{g}}\, \nabla_{\bar{g}} df$ is the non-positive Laplacian acting on the space of functions on $M$.
\end{lemm}
\noindent
A more detailed calculation shows that if $g = \bar{g} + h$ for a metric $\bar{g}$ and a symmetric $(0,2)$-tensor $h$ with $\| h \|_{\bar{g}} << 1$, then
\begin{equation}\label{eq-formula-general}
R_{g} = \bar{R} + DR|_{\bar{g}}(h) + (g+h)^{-1} h g^{-1} h g^{-1} * \mathrm{Ric}_{\bar{g}} + g^{-1} * g^{-1} * g^{-1} * \bar{\nabla} h * \bar{\nabla} h,
\end{equation}
where the term $g^{-1} * g^{-1} * g^{-1} * \bar{\nabla} h * \bar{\nabla} h$ is a contraction of three copies of $g^{-1}$ (i.e., $g$ with raised indices) and two copies of $\bar{\nabla} h = \bar{\nabla} g$.
And, the term $(\bar{g}+h)^{-1} h \bar{g}^{-1} h \bar{g}^{-1} * \mathrm{Ric}_{\bar{g}}$ is the trace of $\mathrm{Ric}_{\bar{g}}$ with respect to $\left( (\bar{g}+h)^{-1} h \bar{g}^{-1} h \bar{g}^{-1} \right)^{-1}$.
Note that $\bar{g} + h$ is positive definite if $\| h\|_{\bar{g}}$ is small enough.
All the proofs of these formulas are given in Section \ref{section-appendix} below.

Take $h = u \cdot g$ for some smooth function $u \in C^{\infty}(M)$ on $M$.
Then for any $s \in \mathbb{R}$,
\begin{equation}\label{eq-formula-conformal}
\begin{split}
    R_{g + t h} &(g + sh) \\
    &= R_{g} g + (DR|_{g} (h) g + R_{g} h) s \\
    &~~~+ \left( \frac{s^{2}u^{2}}{1+su} R_{g} + \frac{s^{2}}{4} (1-su^{-1})^{3} (2n-2)|\nabla u|^{2} \right) (g + sh) \\
    &= R_{g} g -s(n-1)(\Delta_{g} u) g \\
    &~~~+ (1+su) \left( \frac{s^{2}u^{2}}{1+su} R_{g} + \frac{s^{2}}{4} (1-su^{-1})^{3} (2n-2)|\nabla u|^{2} \right)g.
\end{split}
\end{equation}
Hence, if $M$ is a compact manifold with non-empty boundary $\partial M$, then we can take $u$ as the first eigenfunction of $\Delta_{g}$ and obtain the following.
\begin{prop}
    Let $(M^{n}, g)~(n \ge 2)$ be a compact Riemannian $n$-manifold with non-empty boundary $\partial M$.
    Let $u \in C^{\infty}(M)$ be the first eigenfunction of $\Delta_{g}$ with Dirichlet boundary condition.
    Then there is a small $s > 0$ (resp. $s < 0$) such that 
    the metric $g_{s,u} := (1+su) g$ satisfies that
    \begin{equation}\label{eq-scalboundary}
    R_{g_{s, u}} \cdot g_{s, u} > R_{g} \cdot g~~(\mathrm{resp.}~R_{g_{s, u}} \cdot g_{s, u} < R_{g} \cdot g)
    \end{equation}
    at each point in the interior of $M$ and $g_{s, u} = g$ on $\partial M$.
\end{prop}
If, moreover, $\frac{\partial u}{\partial \nu_{g}}$ is positive (resp. negative) everywhere on the boundary $\partial M$, then 
\begin{equation}\label{eq-meanboundary}
H_{g_{s,u}} > H_{g}~(\mathrm{resp.}~H_{g_{s,u}} < H_{g})~~\mathrm{on}~\partial M
\end{equation}
for sufficiently small $s >0$.
Here, $\nu_{g}$ is the unit normal vector field on $\partial M$ of $g$.

Let $x = (x^{1}, \cdots, x^{n}, x^{n+1})$ be the Cartesian coordinate of $\mathbb{R}^{n+1}$ and $\mathbb{S}^{n}_{+} := \{ x \in \mathbb{R}^{n+1} \mid x^{n+1} \ge 0 \} \subset \mathbb{R}^{n+1}$ the upper hemisphere.
Then since the coordinate function $x^{n+1}$ is homogeneous function in $\mathbb{R}^{n+1}$ of degree one, its restriction to $\mathbb{S}^{n}_{+}$ is an eigenfunction of $\Delta_{\delta_{\mathbb{S}^{n}_{+}}}$, whose eigenvalue is $n$.
Here, $\delta_{\mathbb{S}^{n}_{+}}$ is the restriction of the Euclidean metric $\delta$.
Therefore $(\mathbb{S}^{n}_{+}, g_{s,u} := (1+sx^{n+1}|_{\mathbb{S}^{n}_{+}}) \delta|_{\mathbb{S}^{n}_{+}})$ satisfies (\ref{eq-scalboundary}) and (\ref{eq-meanboundary}) for sufficiently small $s > 0$.
On the other hand, in order to construct a similar example of metric $h$ on $\mathbb{S}^{n}_{+}$  satisfying (\ref{eq-meanboundary}) and 
\[
R_{h} > n(n-1) = R_{\delta|_{\mathbb{S}^{n}_{+}}}~\mathrm{on}~\mathbb{S}^{n}_{+}
\]
instead of (\ref{eq-scalboundary}), a more subtle deformation is needed (see {\cite[Theorem 4]{brendle2011deformations}}).
\begin{prop}
    Let $(M^{n}, g)$ be a complete non-compact smooth Riemannian manifold. Assume that $\mathrm{Ric}_{g} \ge -K (n-1)$ for some $K \ge 0$ and $R_{g} > 0$ on $M$.
    Moreover, assume that there is a smooth positive function $u \in C^{\infty}(M)$ satisfying
    \[
    -\Delta_{g} u = \lambda u
    \]
    for some positive constant $\lambda > 0$.
    Then, for any $r > 0$, there is a small $s > 0$ (resp. $s < 0$) depending only on $n, g, \lambda, K$ and $r$ such that $g_{s, u} := (1+su) g$ is a smooth metric in $B_{r}(p)$ and that  
    \[
    R_{g_{s, u}} \cdot g_{s, u} > R_{g} \cdot g~~(\mathrm{resp.}~R_{g_{s, u}} \cdot g_{s, u} < R_{g} \cdot g)
    \]
    at each point of $B_{r}(p)$.
\end{prop}
\begin{proof}
First, we prove the case of $s > 0.$
   From formula (\ref{eq-formula-conformal}) and Li-Wang's gradient estimate {\cite[Theorem 6.1]{li2012geometric}}, there is a constant $C > 0$ depending on $n$ such that 
   \[
   \begin{split}
   R_{g_{s, u}} \cdot g_{s, u} &\ge R_{g} \cdot g + s(n-1)(\lambda u)g + s^{2} (1+su)u^{-1} \left( \frac{R_{g} u^{3}}{1+su} - s^{3} C(r^{-2} + \lambda + K) \right)g \\
   &\ge R_{g} \cdot g + s(n-1)(\lambda u)g \\
   &~~~~+ s^{2} (1+su)u^{-1} \min_{B_{r}(p)} R_{g} \left( \frac{u^{3}}{1+su} - \left( \min_{B_{r}(p)} R_{g} \right)^{-1} s^{3} C(r^{-2} + \lambda + K) \right)g
   \end{split}
   \]
   If $0 \le s \le (\max_{B_{r}(p)} u)^{-1}$, then $(1+su)^{-1} \ge 1/2$.
   So, the desired assertion holds if 
   \[
    0 < s < \min \left\{ \left( \frac{C^{-1}}{2} (r^{-2} + \lambda + K)^{-1} \left( \min_{B_{r}(p)} R_{g} \right) \min_{B_{r}(p)} u^{3} \right)^{1/3}, \frac{1}{\max_{B_{r}(p)} u} \right\}.
   \]
   
   Next, we prove the case of $s < 0$.
   If $0 \ge s \ge -(\max_{B_{r}(p)} u)^{-1}$, then $1+su \ge 0$ and
   \[
   \begin{split}
   R_{g_{s, u}} \cdot g_{s, u} &\le R_{g} \cdot g + s(n-1)(\lambda u)g \\
   &~~~~+ s^{2} u^{2} \left( R_{g} + (1+su)(1-su^{-1})^{3} C(r^{-2} + \lambda + K) \right)g.
   \end{split}
   \]
   If $0 \ge s \ge \max \left\{ - \frac{1}{\min_{B_{r}(p)}u}, -\min_{B_{r}(p)} u \right\} =: s_{r, u}$,
   \[
    R_{g} + (1+su)(1-su^{-1})^{3} C(r^{-2} + \lambda + K) \le \max_{B_{r}(p)} R_{g} + 16C(r^{-2} + \lambda + K).
   \]
   Hence, the desired assertion holds if
   \[
   0 > s > \max \left\{ s_{r, u}, -\frac{(n-1)\lambda}{\max_{B_{r}(p)}u\, (\max_{B_{r}(p)} R_{g} + 16C(r^{-2} + \lambda + K))} \right\}
   \]
   
\end{proof}
\section{Proofs of Main Theorems}\label{section-proof}
From the observation in Section \ref{section-preliminaries}, on every closed $n$-manifold $M$, we cannot deform every metric on $M$ in the conformal direction so that the quantity $R_{g} \cdot g$ increases at each point of $M$. 
Indeed, the first order term (in terms of the parameter $s$) of the perturbed quantity $R_{(1+su)g} (1+su) g$ for a smooth function $u \in C^{\infty}(M)$ is 
\[
-(n-1)(\Delta_{g} u) g.
\]
Hence, one can define the Lipschitz function $\overline{u}$\footnote{From Courant's nodal domain theorem, $\{ x \in M \mid v(x) > 0 \}$ and $\{ x \in M \mid v(x) < 0 \}$ are connected respectively.} as
\[
\overline{u}(x) :=
\begin{cases}
    v(x) &(x \in \{ x \in M \mid v(x) \ge 0 \}), \\
    -v(x) &(x \in \{ x \in M \mid v(x) < 0 \}),
\end{cases}
\]
where $v$ is the first eigenfunction of $-\Delta_g$.
Then, for all sufficiently small $s > 0$, the deformation $\{ g_s := g + s \overline{u} \}$ of $g$ satisfies
\[
R_{g_{s}} \cdot g_s > R_g \cdot g,~R_{g_{s}} > R_g \cdot \| g \|_{1, g_{s}}~\mathrm{and}~R_{g_{s}} > R_g \cdot \| g \|_{2, g_{s}}~~\mathrm{in}~M \setminus N(v),
\]
and each $g_s$ in $\mathrm{Lip}(M) \cap C^{\infty}(M \setminus N(v))$.
Here, $N(v) := \{ x \in M \mid v(x) = 0 \}$ is the nodal set of $v$.
On the other hand, in the smooth category, by the maximum principle, $\Delta_{g} u$ is sign-changing otherwise $u$ is constant on each connected component of $M$.
Therefore, in order to increase the quantity $R_{g} \cdot g$ at each point on the closed manifold, we need to deform the given metric in a direction transverse to the conformal one. 
Let $M$ be a closed manifold and $\mathcal{M}$ the space of all (smooth) Riemannian metrics on $M$.
From \cite{fischer1977manifold}, for any metric $g \in \mathcal{M}$, the tangent space $T_{g}\mathcal{M}$ at $g$ is orthogonally decomposed as
\[
(C^{\infty}(M) \cdot g + \{ \mathcal{L}_{X} g~|~X \in \Gamma(TM) \}) \oplus TT, 
\]
where the subspace $TT$ consists of \textit{tt-tensors} which are trace-free and divergence-free (with respect to $g$) symmetric $(0,2)$-tensors. 
The traceless Ricci tensor
\[
\overset{\circ}{\mathrm{Ric}_{g}} := \mathrm{Ric}_{g} - \frac{R_{g}}{n} g
\]
is orthogonal to the subspace $C^{\infty}(M) \cdot g$ (with respect to $L^{2}(M, g)$ inner product). Moreover, if $R_{g}$ is constant, then it is also a tt-tensor.
Indeed, from the contracted second Bianchi identity:
$g^{jk} R_{ij, k} = (1/2)\, R_{, i}$
and $R_{g} \equiv \mathrm{const}$,
$\overset{\circ}{\mathrm{Ric}_{g}}$ is divergence-free. Since $\overset{\circ}{\mathrm{Ric}_{g}}$ is also trace-free, hence it is a tt-tensor.
We are going to take the tensor $e^{f}\, \overset{\circ}{\mathrm{Ric}_{g}}$ for some $f\ \in C^{\infty}(M)$, which is still trace-free, as the variation $h$ in the following proof of our main theorems \ref{theo-1}, \ref{theo-2}, \ref{theo-3} and \ref{theo-4}.
First, we give a proof of Theorem \ref{theo-3}.
\begin{proof}[Proof of Theorem \ref{theo-3}]
Let $g_{s} = g + s\, e^{f}\, \overset{\circ}{\mathrm{Ric}_{g}}$.
   Then, from the formula (\ref{eq-formula-general}) and the contracted second Bianchi identity, we have 
    \begin{equation*}
    \begin{split}
        &R_{g_{s}} \cdot g_{s} \\
        &= \left( R_{g} + s e^{f} \left[  \left\langle \mathrm{Hess} f + \nabla f \otimes \nabla f, \mathrm{Ric}_{g} - \frac{R_{g}}{n} g \right\rangle_{g} - \| \mathrm{Ric}_{g} \|_{g}^{2} \right. \right.  \\
        &\qquad\left. \left.+ \frac{R^{2}_{g}}{n} + \left( 1 - \frac{1}{n} \right) \langle \nabla f, \nabla R_{g} \rangle _{g} + \left( \frac{1}{2} - \frac{1}{n} \right) \Delta_{g} R_{g} \right] + O(s^{2}) \right) \\
        &\qquad\times \left( g + s\, e^{f} \left( \mathrm{Ric}_{g} - \frac{R_{g}}{n}g \right) \right) \\
        &= R_{g} \cdot g +s e^{f} \left[ \left\{ \left\langle \mathrm{Hess} f + \nabla f \otimes \nabla f, \mathrm{Ric}_{g} - \frac{R_{g}}{n} g \right\rangle_{g} - \| \mathrm{Ric}_{g} \|_{g}^{2} \right. \right. \\
        &\qquad \left. \left. + \left( 1 - \frac{1}{n} \right) \langle \nabla f, \nabla R_{g} \rangle _{g} + \left( \frac{1}{2} - \frac{1}{n} \right) \Delta_{g} R_{g} \right\} g + R_{g} \cdot \mathrm{Ric}_{g} \right]  + O(s^{2})
    \end{split}
    \end{equation*}
    on $TM$.
    Hence the desired assertion follows directly from this formula.
\end{proof}
\begin{proof}[Proof of Theorem \ref{theo-1}]
   Let $g_{s} = g + s\, e^{f}\, \overset{\circ}{\mathrm{Ric}_{g}}$
    Then, as in the above proof of Theorem \ref{theo-3} above, we have
    \[
    \begin{split}
        R_{g_{s} } &- R_{g} \cdot \| g \|^{2}_{1,\, g_{s}}  \\
        &= R_{g} + se^{f} \left[  \left\langle \mathrm{Hess} f + \nabla f \otimes \nabla f, \mathrm{Ric}_{g} - \frac{R_{g}}{n} g \right\rangle_{g} - \| \mathrm{Ric}_{g} \|_{g}^{2} \right. \\
          &\qquad\left.+ \frac{R^{2}_{g}}{n} + \left( 1 - \frac{1}{n} \right) \langle \nabla f, \nabla R_{g} \rangle _{g} + \left( \frac{1}{2} - \frac{1}{n} \right) \Delta_{g} R_{g} \right] + O(s^{2}) \\
        &\qquad- R_{g}\, \frac{1}{1 + s \min_{v \in T_{x}M,\, |v|_{g} = 1}e^{f} \overset{\circ}{\mathrm{Ric}_{g}}(v,v)} \\
        &=R_{g} + se^{f} \left[  \left\langle \mathrm{Hess} f + \nabla f \otimes \nabla f, \mathrm{Ric}_{g} - \frac{R_{g}}{n} g \right\rangle_{g} - \| \mathrm{Ric}_{g} \|_{g}^{2} \right. \\
          &\qquad\left.+ \frac{R^{2}_{g}}{n} + \left( 1 - \frac{1}{n} \right) \langle \nabla f, \nabla R_{g} \rangle _{g} + \left( \frac{1}{2} - \frac{1}{n} \right) \Delta_{g} R_{g} \right] + O(s^{2}) \\
        &\qquad- R_{g} \left( 1 - s \min_{v \in T_{x}M,\, |v|_{g} = 1}e^{f} \overset{\circ}{\mathrm{Ric}_{g}}(v,v) + O(s^{2}) \right).
    \end{split}
    \]
    Then the assertion immediately follows from this.
\end{proof}
\begin{proof}[Proof of Theorem \ref{theo-2}]
Let $g_{s} = g + s\, e^{f}\, \overset{\circ}{\mathrm{Ric}_{g}}$
    Then, as in the above proof of Theorem \ref{theo-1}, we have
\begin{equation}\label{eq-proof-theo-2-1}
    \begin{split}
    R_{g_{s}} &- R_{g} \cdot \| g \|_{2, g_{s}} \\
    &= R_{g} + se^{f} \left[  \left\langle \mathrm{Hess} f + \nabla f \otimes \nabla f, \mathrm{Ric}_{g} - \frac{R_{g}}{n} g \right\rangle_{g} - \| \mathrm{Ric}_{g} \|_{g}^{2} \right. \\
          &\qquad\left.+ \frac{R^{2}_{g}}{n} + \left( 1 - \frac{1}{n} \right) \langle \nabla f, \nabla R_{g} \rangle _{g} + \left( \frac{1}{2} - \frac{1}{n} \right) \Delta_{g} R_{g} \right] + O(s^{2}) \\
        &\qquad- R_{g}\, \sqrt{\max_{v \in \Lambda^{2}T_{x}M \setminus \{ 0 \}} \frac{g(v,v)}{\left( g + s\, e^{f}\, \overset{\circ}{\mathrm{Ric}_{g}} \right)(v,v)}}.
    \end{split}
    \end{equation}
    Here, for any regular $n \times n$ matrix $A$ and a $n \times n$ matrix $B$, 
    \[
    \det{(A +t B)} =
    \det{A} \cdot \det{\left( \mathrm{Id}_{n} + t A^{-1} B \right)}
    = \det{A} \left( 1 + t\, \mathrm{trace} (A^{-1} B) + O(t^{2}) \right).
    \]
    So, we have
    \begin{equation}\label{eq-proof-theo-2-2}
    \begin{split}
    \sqrt{\max_{v \in \Lambda^{2}T_{x}M \setminus \{ 0 \}} \frac{g(v,v)}{\left( g + s\, e^{f}\, \overset{\circ}{\mathrm{Ric}_{g}} \right)(v,v)}}
    &= 1 - \frac{s}{2} \min_{v \in \Lambda^{2} T_{x} M \setminus \{ 0 \}} e^{f} g^{-1} \overset{\circ}{\mathrm{Ric}_{g}}(v,v) + O(s^{2}).
    \end{split}
    \end{equation}
Putting (\ref{eq-proof-theo-2-1}) and (\ref{eq-proof-theo-2-2}) together, we can immediately get the assertion. 
\end{proof}

\begin{proof}[Proof of Theorem \ref{theo-4}]
    As the proof of Theorem \ref{theo-2}, one can also compute the difference 
    $R_{g_{s}}^{2} \cdot g - R_{g}^{2} \cdot g$ on $\Lambda^{2} TM$.
    Then, from the above calculation in the proof of Theorem \ref{theo-2}, it is easy to see that the assertion holds .
\end{proof}
Next, we give a proof of Theorem \ref{theo-main-unstable}.
In the proof, we consider another trace-free tensor: the traceless Hessian $\nabla^2 u - \frac{\Delta_g u}{n} g$ of a function $u \in C^{\infty}(M)$.
\begin{proof}[Proof of Theorem \ref{theo-main-unstable}]
Let $h = \nabla^2_{g} u - \frac{\Delta_g u}{n} g$, then
\[
\begin{split}
DR|_{g}(h) &= \mathrm{div}_g \left( \mathrm{div}_g \left( \nabla^2_{g} u - \frac{\Delta_g u}{n} g \right) \right) - \left\langle \mathrm{Ric}_g, \nabla^2_{g} u - \frac{\Delta_g u}{n} g \right\rangle \\
&= \left( 1- \frac{1}{n} \right) \Delta_g^{2} u + \frac{1}{2} \langle \nabla R_g, \nabla u \rangle + \frac{\Delta_g u}{n} R_g
\end{split}
\]
by Lemma \ref{lemm-scal-variation} and the Ricci identity.
Hence, from the formula (\ref{eq-formula-general}) and the assumption $R_g = \mathrm{const}$ (cf. proof of Theorem \ref{theo-1} above), we have for $g_s = g -s h$ and at $x \in M$,
\begin{equation}\label{eq-obata-deformation}
\begin{split}
  R_{g_s} - R_g \cdot \| g \|_{1, g_s}^2 &= -s \left[ \left( \left( 1- \frac{1}{n} \right) \Delta_g^{2} u + \frac{\Delta_g u}{n} R_g \right) \right. \\
 &\qquad+ \left. R_g\, \min_{v \in T_x M,\, |v|_g=1} \left( \nabla_g^2 u - \frac{\Delta_g u}{n}g \right)(v,v) \right] + o(s).
\end{split}
\end{equation}
First, we assume that $g$ is unstable, thus $R_g$ is a positive constant.
Let $v$ the first eigenfunction of $-\Delta_g$ with the first eigenvalue $\lambda_1$.
Then, our assumption implies that $ 0 < \lambda_1 < \frac{R_g}{n-1}$.
One can decompose $M$ into
\[
M = M(v)^{+} \sqcup M(v)^{-} \sqcup N(v),
\]
where
$M(v)^{+} = v^{-1}((0, +\infty))$, $M(v)^{-} = v^{-1}((-\infty, 0))$ and $N(v) = v^{-1}(\{ 0 \})$.
Here, $M^{\pm}(v)$ are open subset of $M$ and $N(v)$ is a closed subset of $M$.
Moreover, Courant's nodal domain theorem states that $M^{\pm}(v)$ are both connected.
Then, since $-\Delta_g v = \lambda_1 v$ and hence $\int_M v\, d\mathrm{vol}_g = 0$, there is a smooth function $\bar{u} \in C^{\infty}(M)$ such that $\Delta_g \bar{u} = v$ on $M$.
Then, define the function $u \in C^{\infty}(M \setminus N(v))$ as
\[
u :=
\begin{cases}
    -\bar{u} &\mathrm{on}~M(v)^{+}, \\
    \bar{u} &\mathrm{on}~M(v)^{-}.
\end{cases}
\]
Then, by the choice of the function $u$ and (\ref{eq-obata-deformation}) above, we can obtain the desired assertion for the type I scalar curvature rigidity. 
Note that since $\nabla_g^2 u - \frac{\Delta_g u}{n}g$ is trace free (with respect to $g$),
\[
 \min_{v \in T_x M,\, |v|_g=1} \left( \nabla_g^2 u - \frac{\Delta_g u}{n}g \right)(v,v) \le 0
\]
at each point $x \in M$.
Next, if $g$ is stable and $R_g \le 0$, by replacing $s$ with $-s$, we can show the assertion in the same way.
Finally, if $g$ is stable and $R_g$ is a positive constant, by replacing $u$ with $-u$, we can show the assertion in the same way.

Proof for the type II scalar curvature rigidity is the same (cf. proof of Theorem \ref{theo-2} above).
\end{proof}
\begin{rema}
    If the nodal set $N(v)$ is smooth enough so that there are solutions $u_1 \in C^{\infty}(M^{+}(v)) \cap C(\overline{M^{+}(v)}),\, u_2 \in C^{\infty}(M^{-}(v)) \cap C(\overline{M^{-}(v)})$ of the Laplace equations:
    \[
    \begin{cases}
        \Delta_g u_1 = 0~\mathrm{in}~M(v)^{+},~~u_1|_{N(v)} = -\bar{u}|_{N(v)}, \\
        \Delta_g u_2 = 0~\mathrm{in}~M(v)^{-},~~u_2|_{N(v)} = \bar{u}|_{N(v)},
    \end{cases}
    \]
    then we can take the function $u \in C^{\infty}(M \setminus N(v))$ as
\[
u :=
\begin{cases}
    -\bar{u} - u_1 &\mathrm{on}~M(v)^{+}, \\
    \bar{u} - u_2 &\mathrm{on}~M(v)^{-}, \\
    0 &\mathrm{on}~N(v).
\end{cases}
\]
Then, $u$ is in $\mathrm{Lip}(M)$ and $\Delta_g u = \mp \Delta_g \overline{u}$ on $M(v)^{\pm}$.

If the degenerate part $S(v) := \{ x \in N(v) \mid |\nabla v| = 0 \}$ of $N(v)$ is finite, then one can construct solutions $u_1, u_2$ as above.
Indeed, let $S(v) = \{ x_1, \cdots, x_k \}$. 
From {\cite[Theorem 2.2]{cheng1976eigenfunctions}}, $N(v) \setminus S(v)$ is smooth $(n-1)$-dimensional submanifold of $M$.
Take a family of disjoint geodesic balls $\{ B_{\rho_{j}} (x_k) \}_{i = 1}^{k}$ centered at each $x_i~(i = 1, \cdots, k)$ with sufficiently small radius $\rho_j > 0$, where $\{ \rho_j \}$ is a decreasing sequence of positive real numbers.
Consider $(M(v)^{+} \setminus \cup^{k}_{i=1} B_{\rho_j}(x_i)) \cup N(v) \cup (\cup^{k}_{i=1} \partial B_{\rho_j}(x_i))$ for each $j$ and an appropriate smoothing it, we obtain the family of smooth (up to the boundary) domains $\{ \Omega_j \}$ which exhausts $M(v)^{+} \cup N(v) \setminus S(v)$.
Let $u_j$ be the solution of 
\[
\begin{cases}
    \Delta_g u_j = 0 &\mathrm{in}~\Omega_j \\
    u_j|_{\partial \Omega_j} = -\overline{u}|_{\partial \Omega_j} &\mathrm{on}~\partial \Omega_j.
\end{cases}
\]
Then, since $\overline{u}$ is smooth on $M$, we have a $C^{2, \alpha}$-estimate up to the boundary
\[
\| u_j \|_{C^{2, \alpha}(\overline{\Omega_j})} \le C
\]
for some $C = C(M^n, v) > 0$ and some $\alpha \in (0, 1)$.
Here, $C$ is independent of $j$.
Thus, by taking a subsequece, we can obtain a function $\overline{u}_1 \in C^{2, \beta}(M(v)^{+} \cup N(v) \setminus S(v))$ such that
$\Delta_g \overline{u}_1 = 0$ on $M(v)^{+} \cup N(v) \setminus S(v)$.
From the construction of $\{ u_j \}$, if we set
\[
u_1(x) :=
\begin{cases}
\overline{u}_1(x) & x \in M(v)^{+} \cup N(v) \setminus S(v) \\
 -\overline{u}(x)  &x = x_i~(i = 1, 2, \cdots, k),
\end{cases}
\]
then $u_1 \in C(M(v)^{+} \cup N(v))$ and this is the desired solution on $M(v)^{+} \cup N(v)$.
The construction of $u_2$ is the same.

\bigskip
More generally, if $S(v)$ is approximated by a family $\{ \Omega^{+}_j \}$ (resp. $\{ \Omega^{-}_{j} \}$) of smooth domains (up to boundary) which exhausts $M(v)^{+} \cup N(v) \setminus S(v)$ (resp. $M(v)^{-} \cup N(v) \setminus S(v)$) and converges to $M(v)^{+} \cup N(v) \setminus S(v)$ (resp. $M(v)^{-} \cup N(v) \setminus S(v)$) uniformly, i.e. for any $\varepsilon > 0$, there exists $j_0  = j_0(\varepsilon)$ such that $\sup \{ d(x, S(v)) \mid x \in \partial \Omega_j \} < \varepsilon$ for all $j \ge j_0$, then one can construct solutions $u_1, u_2$ as above in the same way as above.
However, in general, all we know is that $S(v)$ is a closed countably $(n-2)$-rectifiable subset of $M$ \cite{hardt1989nodal}.
\end{rema}
\begin{rema}
    From the proof of Theorem \ref{theo-main-unstable} above,
    one can also observe that
    \[
    R_{g_s} > R_g~\mathrm{on}~M \setminus N(v)
    \]
    for all sufficiently small $s > 0$.
    Here, the family of the metrics $\{ g_s \}$ is the same as the one in the proof of Theorem \ref{theo-main-unstable} above. (Hence, each $g_s$ is smooth only on $M \setminus N(v)$.)
    However, of course, the ``size'' information such as length, volume, etc. may vary in general.
\end{rema}
Using the same deformation, we can also show the following.
\begin{prop}\label{prop-local-obata}
    Let $(M^n, g)~(n \ge 3)$ be a closed Riemannian manifold whose scalar curvature $R_g$ is a nonzero constant.
    If $\frac{R_g}{n-1}$ is an eigenvalue of $-\Delta_g$, then for any corresponding eigenfunction $u \in C^{\infty}(M)$, 
    \[
    U(u) := \left\{ x \in M \left|~ \left\| \nabla^{2}_g u - \frac{\Delta_g u}{n}g \right\|(x) \neq 0 \right. \right\}
    \]
    and any compact set $K \subset U(u)$, there is a positive small constant $s_0 = s_0(K) > 0$ such that
    for all $0 < s \le s_0$, $g_s = g - \mathrm{sign}(R_g)\,  s \left( \nabla^2_g u - \frac{\Delta_g u}{n}g \right)$ is a Riemannian metric on $M$ and 
    \[
R_{g_s} > R_g \cdot \| g \|^2_{1, g_s},~~R_{g_s} > R_g \cdot \| g \|_{2, g_s}~~\mathrm{on}~K.
    \]
\end{prop}
\begin{rema}
From {\cite[Theorem 24]{kuhnel1988conformal}} (cf. \cite{tashiro1965complete, tashiro1967conformal}), if $(M, g)$ satisfies the assumption of Proposition \ref{prop-local-obata} and is not isometric to the standard sphere of certain radius, then $U(u)$ is always nonempty for any $u \in C^{\infty}(M)$.
\end{rema}
By taking the first eigenfunction of the conformal Laplacian $-\Delta_{g} + \frac{n-2}{4(n-1)} R_g$ as $v$ instead in the proof of Theorem \ref{theo-main-unstable}, we can similarly prove the following:
\begin{theo}
    Let $(M^n, g)~(n \ge 3)$ be a closed Riemannian manifold and $\lambda_1$ be the first eigenvalue of the conformal Laplacian $-\Delta_g + \frac{n-2}{4(n-1)}R_g$.
    Assume the scalar curvature $R_g$ is constant.
    Then the following holds:
    \begin{itemize}
        \item If $R_g < 0$ and $\lambda_1 < \frac{n+2}{4(n-1)}R_g$, then there is a sufficiently small $s_0 > 0$ and a function $u \in C^{\infty}(M)$ such that
        $g_s := g +s\left( \nabla^2_{g} u - \frac{\Delta_g}{n} g \right)$ is a Riemannian metric, and
        \[
        R_{g_s} > R_g \cdot \| g \|^2_{1, g_s}~\mathrm{and}~R_{g_s} > R_g \cdot \| g \|_{2, g_s}~\mathrm{on}~M
        \]
        for all $s \in (0, s_0]$.
        \item If $R_g > 0$ and $\lambda_1 > \frac{n+2}{4(n-1)}R_g$, then there is a sufficiently small $s_0 > 0$ and a function $u \in C^{\infty}(M)$ such that
        $g_s := g -s\left( \nabla^2_{g} u - \frac{\Delta_g}{n} g \right)$ is a Riemannian metric, and
        \[
        R_{g_s} > R_g \cdot \| g \|^2_{1, g_s}~\mathrm{and}~R_{g_s} > R_g \cdot \| g \|_{2, g_s}~\mathrm{on}~M
        \]
        for all $s \in (0, s_0]$.
    \end{itemize}
\end{theo}
\begin{proof}
    Let $v \in C^{\infty}(M)$ be the positive first eigenfunction of the conformal Laplacian $-\Delta_g + \frac{n-2}{4(n-1)}R_g$.
    Let $u \in C^{\infty}(M)$ be the solution of $\Delta_g u = v - \int_M v\, d\mathrm{vol}_g$.
    Set $g_s = g -s \left( \nabla^2_{g} u - \frac{\Delta_g u}{n}g \right)$, then it becomes a Riemannian metric on $M$ for each sufficiently small $s > 0$.
    Since $R_g = \mathrm{const}$, from the calculation of the proof of Theorem \ref{theo-cscdeformation}, we have
    \[
    \begin{split}
    R_{g_s} - R_g \cdot \| g \|^2_{1, g_s} &= -s \left[ \left( -\frac{n-1}{n}\lambda_1 + \frac{n+2}{4n}R_g \right)v - \frac{R_g}{n} \left( \int_M v\, d\mathrm{vol}_g \right) \right. \\
    &\qquad\left. + R_g \min_{v \in T_X M,\, |v|_g = 1} \left( \nabla^2_{g} u - \frac{\Delta_g u}{n} g \right) \right] + o(s) \\
    &=  -s \left[ \left( -\frac{n-1}{n}\lambda_1 + \frac{n+2}{4n}R_g \right)v \right. \\
    &\qquad\left. + R_g \min_{v \in T_X M,\, |v|_g = 1} \left( \nabla^2_{g} u - \frac{v}{n} g \right) \right] + o(s).
    \end{split}
    \]
    Here, the trace of $\nabla^2_{g} u - \frac{v}{n} g$ with respect to $g$ is 
    $-\int_M v\, d\mathrm{vol}_g < 0$.
    Hence, 
    \[
    \min_{v \in T_X M,\, |v|_g = 1} \left( \nabla^2_{g} u - \frac{v}{n} g \right) < 0
    \]
    at each $x \in M$.
    Therefore, when $R_g = \mathrm{const} > 0$, we can show the desired assertion if we take $s_0 > 0$ sufficiently small.
    When $R_g = \mathrm{const} < 0$, by replacing $s$ with $-s$, the same argument above leads to the assertion.
    Proof for the type II scalar curvature rigidity is the same.
\end{proof}
Finally, we give a proof of Theorem \ref{theo-cscdeformation}.
\begin{proof}[Proof of Theorem \ref{theo-cscdeformation}]
According to \cite{anderson2005uniqueness},
if $M^n~(n \ge 3)$ is a closed manifold and $g$ is a strongly stable unique non-negative Yamabe metric with unit volume in its conformal class (then every metric sufficiently $C^{\infty}$-close to $g$ also contains the unique Yamabe metric in its conformal class), then for any $(0,2)$-tensor $h$ with $\mathrm{tr}_{g} h = 0$, we have
  \[
  R_{\gamma_{t}} \cdot \gamma_{t} - R_{g} \cdot g = R_{g} (\gamma_{t} - g) + t \left( - \int_{M} \langle \overset{\circ}{\mathrm{Ric}}_{g}, h \rangle\, d\mathrm{vol}_{g} \right) \gamma_{t} + o(t)
  \]
  for all sufficiently small $t > 0$.
  Here, $\gamma_{t} \in [g_{t} := g + th]$ is the unique Yamabe metric in its conformal class with unit volume.
Since $\gamma_{t}$ is the unique Yamabe (hence constant scalar curvature) metric in $[\gamma_{t}]$ for all sufficiently small $t > 0$, 
\[
\gamma_{t} - g = t\, \mathrm{pr}_{T_{g}\mathcal{C}_1} (h) + o(t),
\]
where $\mathrm{pr}_{T_{g}\mathcal{C}_1}$ is the projection onto the tangent space of $\mathcal{C}_1$ at $g$, where $\mathcal{C}_1$ is the space of all Riemannian metrics with constant scalar curvature and unit volume.
In particular, if $g$ is an Einstein metric and $h$ is a tt-tensor with respect to g, then $\mathrm{pr}_{T_{g}\mathcal{C}}(h) = h$.
Therefore, if we take $h = -\overset{\circ}{\mathrm{Ric}}_{g}$, then we have
\[
\begin{split}
R_{\gamma_{t}} \cdot \gamma_{t} - R_{g} \cdot g &= t \left( \| \overset{\circ}{\mathrm{Ric}}_{g} \|^{2}_{L^{2}(M, g)} g - R_{g} \cdot \mathrm{pr}_{T_{g}\mathcal{C}}(\overset{\circ}{\mathrm{Ric}}_{g}) \right) + o(t).
\end{split}
\]
Then, the assertion follows from this formula (cf. proof of Theorem \ref{theo-1} and \ref{theo-2}).
\end{proof}
 By {\cite[Corollary 2]{matsuo2014prescribed}}, for every closed manifold $M$ of dimension $\ge 3$ with positive Yamabe invariant, there is a non-Ricci-flat scalar-flat metric $g$ on $M$.
  Thus, the above argument with $h = -\mathrm{Ric}_{g}$ for such $g$ implies that $R_{\gamma_{t}} \cdot \gamma_t > 0$ for any sufficiently small $t > 0$ where $\gamma_{t} \in [g + t\, \mathrm{Ric}_{g}]$ is the unique unit-volume Yamabe metric in its conformal class.
  Therefore, such a metric $g$ is not a local maximizer of the functional : $\mathcal{M}_{1} \ni h \mapsto Y(M, [h])$, and neither type I, II nor III scalar curvature rigid in the sense of Listing.

  \begin{ques}
     Let $M$ be a smooth manifold of positive Yamabe invariant $Y(M) > 0$.
     For a strongly stable non-Einstein unique Yamabe metric $h$ on $M$ with unit volume, is there a positive Yamabe metric $g$ with unit volume for which the inequality in (\ref{eq-kato}) is strict?
  \end{ques}

\section{Examples}\label{section-examples}
\subsection{On the product of two Einstein manifolds}
Let $m, n \ge 2$ and $(M^{m}, g_{M})$, $(N^{n}, g_{N})$ be closed Riemannian manifolds with constant curvature $1$ and $-1$ respectively.
Consider the metric of the form
\[
g_{\lambda} := g_{M} + \lambda g_{N},
\]
where $\lambda \ge 1$ is a scaling constant. 
Then, for each point $x \in M \times N$,
\[
\mathrm{Ric}(g_{\lambda}) = 
\begin{cases}
    (m-1) g_{M^{m}} & \mathrm{on}~T_{\mathbf{p}_{1} x} M^{m} \subset T_{x} (M^{m} \times N^{n}), \\
    -(n-1) g_{N^{n}} &  \mathrm{on}~T_{\mathbf{p}_{2} x} N^{n} \subset T_{x} (M^{m} \times N^{n}),
\end{cases}
\]
and hence
\[
R(g_{\lambda}) = m(m-1) -\frac{n(n-1)}{\lambda}.
\]
Here, $\mathbf{p}_{i}~(i = 1,2)$ denotes the natural projection from, respectively $M^{m}$ and $N^{n}$ to $M^{n} \times N^{n}$. 

So, if we put $n = m$, then one can check that $R(g_{\lambda}) \ge 0$ and $g_{\lambda}$ satisfies the assumptions of Theorems \ref{theo-1}, \ref{theo-2}, \ref{theo-3} and \ref{theo-4} for all $\lambda \ge 1$.
Moreover, if we put $m = n+1$, one can also check that $R(g_{1}) \ge 0$ and $g_{1}$ satisfies the assumptions of Theorems \ref{theo-1}, \ref{theo-2}, \ref{theo-3} and \ref{theo-4}.
On the other hand, we can directly check that $g_{1}$ is none of type I, II, III or IV scalar curvature rigid in the sense of Listing.
In fact, the family $\{ g_{\lambda} \}_{\lambda \ge 1}$ gives a deformation which suggests that $g_{1}$ is not rigid in each sense. 
In particular, for every dimension $n \ge 4$, there is an $n$-dimensional manifold on which there is a non-rigid metric $g$ in each sense.

\subsection{Examples on Lie groups}\label{subsection-lie}
We can construct examples of left invariant metrics on some three dimensional Lie groups that satisfy the assumptions of our main theorems in Section \ref{section-introduction}.
Curvatures of left invariant metrics on Lie groups have been studied by Milnor in \cite{milnor1976curvatures}.
Especially, the case of three-dimensional unimodular is written in Section 4 of \cite{milnor1976curvatures}.
The following are examples of metrics on unimodular three-dimensional Lie groups, which satisfy the assumption of Theorem \ref{theo-1} or \ref{theo-3}.
\begin{itemize}
    \item A left invariant metric on $SU(2)$ (which is homeomorphic to the unit 3-sphere $\mathbb{S}^{3}$), whose signature of Ricci quadratic form is $(+, -, -)$
    (see also {\cite[Example 2]{shin1994examples}}).
    More specifically, we consider the Berger sphere:
    \[
    \begin{split}
        \mathbb{S}^{3}(1) \cong SU(2) 
        &= \left\{ A \in M(2, \mathbb{C}) \mid~\det A =1,~A^{*} = -A \right\} \\
        &= \left\{ \left. \begin{pmatrix}
            z &-\bar{w} \\
            w &\bar{z}
            \end{pmatrix} \right|~(z, w) \in \mathbb{C}^{2},~|z|^{2} + |w|^{2} = 1
        \right\},
    \end{split}
    \]
    and set 
    \[
    X_{1} =
    \begin{pmatrix}
        \sqrt{-1} &0 \\
        0 &-\sqrt{-1}
    \end{pmatrix},~
    X_{2} =
    \begin{pmatrix}
        0 &1 \\
        -1 &0
    \end{pmatrix},~
    X_{3} =
    \begin{pmatrix}
        0 &\sqrt{-1} \\
        \sqrt{-1} &0
    \end{pmatrix}.
    \]
    We define the left invariant metric $g_{s,t}~(1 \le s \le t)$ on $SU(2)$ so that
    \[
    g_{s,t}(X_{1}, X_{1}) = 1, g_{s,t}(X_{2}, X_{2}) = s, g_{s,t}(X_{3}, X_{3}) = t, g_{s,t}(X_{i}, X_{j}) = 0~(i \neq j).
    \]
    Then, using the orthonormal frame 
    \[
    v_{1} := X_{1},~v_{2} := \frac{1}{\sqrt{s}} X_{2},~v_{3}:= \frac{1}{\sqrt{t}} X_{3},
    \]
    one can compute the Ricci and scalar curvatures as follows.
    \[
    \begin{split}
    \mathrm{Ric}_{g_{s,t}} (v_{1}, v_{1}) &= -\frac{1}{st} (-2 +2t^{2} + 2s^{2} -4st),  \\
\mathrm{Ric}_{g_{s,t}}(v_{2}, v_{2}) &= -\frac{1}{st} (2 + 2t^{2} -2s^{2} -4t), \\
\mathrm{Ric}_{g_{s,t}} (v_{3}, v_{3}) &= -\frac{1}{st} (2 -2t^{2} + 2s^{2} -4s),
    \end{split}
    \]
    and
    \[
    R_{g_{s,t}} = \frac{2}{st}\{ 2(s+t+st) - (1+s^{2} +t^{2}) \}.
    \]
   For example, if we consider the case of $(s, t) = (1, 4 - \varepsilon/2)~(\varepsilon \le 6)$ and set $g_{\varepsilon} := g_{1, 4 - \varepsilon/2}$, then
        \[
        \mathrm{Ric}_{g_{\varepsilon}}(v_{i}, v_{i}) = -4 + \varepsilon~(i = 1,2),~\mathrm{Ric}_{g_{\varepsilon}}(v_{3}, v_{3}) = 8-\varepsilon,
        \]
        and 
        \[
        \| \overset{\circ}{\mathrm{Ric}}_{g_{\varepsilon}} \|_{g_{\varepsilon}}^{2} = \frac{3}{2} \left( 8 - \frac{4}{3}\varepsilon \right)^{2},~R_{g_{\varepsilon}} \left( \overline{\lambda}_{Ric} - \frac{R_{g_{\varepsilon}}}{3} \right) = \varepsilon \left( 8 - \frac{4}{3}\varepsilon \right).
        \]
        Note here that $g_{\varepsilon}$ is the standard metric on $\mathbb{S}^3(1)$ with constant curvature if and only if $\varepsilon = 6$.
        Hence, if $0 < \varepsilon < 4$, then $g_{\varepsilon}$ has a positive constant scalar curvature and satisfies the assumptions of Theorems \ref{theo-1}, \ref{theo-2}, \ref{theo-3} and \ref{theo-4}.
        On the other hand, one can directly check that for any $\varepsilon_0 < 4$, $R_{g_{\varepsilon}} \cdot g_{\varepsilon} \ge R_{\varepsilon_0} \cdot g_{\varepsilon_0}$ on $T\mathbb{S}^3$ for all $\varepsilon_0 \le \varepsilon < 4$. 
        Moreover, from {\cite[Example 2]{shin1994examples}}, if $\varepsilon \le 2$ (resp. $\varepsilon < 2$), then $g_{\varepsilon}$ is a (resp. unique) Yamabe metric with positive scalar curvature. 
        On the other hand, if $\varepsilon \le 0$, then $g_{\varepsilon}$ has a non-positive constant scalar curvature and satisfies the assumptions of Theorems \ref{theo-1}, \ref{theo-2} and \ref{theo-3} but does not one of Theorem \ref{theo-4}.
        On the other hand, there is $\varepsilon_0 \in \left( \frac{11}{2}, \frac{17}{3} \right)$, it holds that
        \[
       \| \overset{\circ}{\mathrm{Ric}}_{g_{\varepsilon}} \|_{L^2}\, g_{\varepsilon} - R_{g_{\varepsilon}}\, \overset{\circ}{\mathrm{Ric}}_{g_{\varepsilon}} > 0~~\mathrm{on}~M
        \]
        for all $\varepsilon < \varepsilon_0$.
        Hence, if $g_{\varepsilon}$ with $2 \le \varepsilon < \varepsilon_0$ is a unique Yamabe metric (up to rescaling) in its conformal class\footnote{Since $Y(\mathbb{S}^3, [g_{\varepsilon}]) = \varepsilon \left( 4 - \frac{\varepsilon}{2} \right)^{1/3} \mathrm{Vol}(\mathbb{S}^3, g_{std})^{2/3} < 6\, \mathrm{Vol}(\mathbb{S}^3, g_{std})^{2/3} = Y(\mathbb{S}^3, [g_{6}] = [g_{std}])$, $(M, [g_{\varepsilon}])$ is not conformally equivalent to the standard sphere.}, then such a metric $g_{\varepsilon}$ can be deformed in the csc direction so that neither Type I nor II rigidities hold by Theorem \ref{theo-cscdeformation}.
    \item Any left invariant metric on the Heisenberg group 
    \[
    \left\{ \begin{bmatrix} 
    1 &* &* \\
    0 &1 &* \\
    0 &0 &1
    \end{bmatrix} \in \mathbb{M}(\mathbb{R}, 3) 
    \right\},
    \]
    whose signature of Ricci quadratic form is $(+, -, -)$.
    \item A left invariant metric on $SL(2, \mathbb{R})$ or $E(1,1)$, whose signature of Ricci quadratic form can be either $(+, -, -)$ or $(0, 0, -)$ depending on the choice of 
 the left invariant metric. 
\end{itemize}
For each metric with constant negative scalar curvature in the above example, from \cite{koiso1979decomposition}, 
\[
\mathcal{C}_{1} := \{ g \in \mathcal{M} \mid~R_{g} = \mathrm{const},~\mathrm{Vol}(M, g)=1 \}
\]
is a (infinite dimensional) manifold near such a metric (after normalizing it so that it has unit volume). 
Here, $\mathcal{M}$ is the space of all Riemannian metrics on each manifold $M$.
Moreover, the condition that 
\[
\| \overset{\circ}{\mathrm{Ric}_{g}} \|_{g}(x) \neq 0~\mathrm{for~all}~x \in M
\]
is an open condition with respect to the $C^{\infty}$-topology on $\mathcal{M}$.
Therefore, all metrics in $\mathcal{C}_{1}$ sufficiently $C^{\infty}$-close to such a metric also satisfy the assumption of Theorem \ref{theo-1}.

\subsection{Examples on the total spaces of Riemannian submersions with totally geodesic fibers}\label{subsection-submersions}
\begin{itemize}
    \item ({\cite[Section 5]{matsuzawa1983einstein}} or {\cite[9.82 Example 2]{besse2007einstein}})~Let $\pi : S^{4n+3} \rightarrow \mathbb{H}P^{n}$ be the Hopf fibration whose fibers are the standard unit $3$-sphere $\mathbb{S}^{3}$.
    We denote the scalar curvatures of the fibers, the base space and the total space by $R^{F}, R^{B}$ and $R^{M}$, respectively.
    Then, 
    \[
    R^{F} = 6,~R^{B} = 16n(n+2),~R^{M} = (4n+3)(4n+2).
    \]
    The Ricci curvature of the canonical variation $g_{t}~(t > 0)$ was calculated in \cite{matsuzawa1983einstein} as follows.
    \[
    \begin{split}
        \mathrm{Ric}^{t}(U,V) &= \left( \frac{R^{F}}{t\,  \mathrm{dim}F} + t \left( \frac{R^{M}}{\mathrm{dim}F + \mathrm{dim}B} - \frac{R^{F}}{\mathrm{dim}F} \right) \right)\, g_{t}(U, V), \\
        \mathrm{Ric}^{t}(X,Y) &= \left( \frac{R^{B}}{\mathrm{dim}B} + t \left( \frac{R^{M}}{\mathrm{dim}F + \mathrm{dim}B} - \frac{R^{B}}{\mathrm{dim}B} \right) \right)\, g_{t}(X, Y),
    \end{split}
    \]
    where $U, V$ are vertical vectors and $X, Y$ are horizontal vectors.
    So, in this example, 
    \[
    \mathrm{Ric}^{t}(U,V) = \left( \frac{2}{t} + 4nt \right)\, g_{t}(U, V), ~
        \mathrm{Ric}^{t}(X,Y) = \left( 4(n+2) -6t  \right)\, g_{t}(X, Y)
    \]
    and 
    \[
    R_{g_{t}} = \frac{6}{t} -12nt +16n(n+2).
    \]
    Therefore, 
    \[
    \begin{split}
    \| \overset{\circ}{\mathrm{Ric}_{g_{t}}} \|^{2}_{g_{t}} &= 3 \left( \frac{2}{t} + 4nt -\frac{1}{4n+3}\left( \frac{6}{t} -12nt +16n(n+2) \right) \right)^{2} \\
    &+ 4n \left( 4(n+2) -6t -\frac{1}{4n+3}\left( \frac{6}{t} -12nt +16n(n+2) \right) \right)^{2}\\
    &= 3 \left( \frac{8n}{4n+3} \cdot \frac{1}{t} + \frac{4n(4n+6)}{4n+3} t - \frac{16n(n+2)}{4n+3} \right)^{2} \\
&+ 4n \left( \frac{6}{4n+3} \cdot \frac{1}{t} + \frac{12n+18}{4n+3} t - \frac{12(n+2)}{4n+3} \right)^{2}.
    \end{split}
    \]
    Let $\lambda^{V}_{Ric_{t}}$ and $\lambda^{H}_{Ric_{t}}$ be the eigenvalues of $\mathrm{Ric}_{g_{t}}$ in the vertical and horizontal directions, respectively.
    One can observe that 
    \[
    \begin{cases}
        \lambda_{Ric_{t}}^{V} \ge \lambda_{Ric_{t}}^{H} &0 < t \le \frac{1}{2n+3}~\mathrm{or}~t \ge 1 \\
         \lambda_{Ric_{t}}^{H} \ge \lambda_{Ric_{t}}^{V} &\frac{1}{2n+3} \le t \le 1.
    \end{cases}
    \]
    Matsuzawa \cite{matsuzawa1983einstein} observed that the canonical variation $g_{t}$ is an Einstein metric on $\mathbb{S}^{4n+3}$ if and only if $t =1$ or $t = \frac{1}{2n+3}$.
    Then, we have
    \[
    \begin{split}
        &R_{g_{t}} \left( \lambda_{Ric_{t}}^{V} - \frac{R_{g_{t}}}{\mathrm{dim}M} \right) \\
        &= \left( \frac{6}{t} -12nt + 16n(n+2) \right)\left( \frac{8n}{4n+3} \cdot \frac{1}{t} + \frac{4n(4n+6)}{4n+3} t - \frac{16n(n+2)}{4n+3} \right), \\
         &R_{g_{t}} \left( \lambda_{Ric_{t}}^{H} - \frac{R_{g_{t}}}{\mathrm{dim}M} \right) \\
         &= \left( \frac{6}{t} -12nt + 16n(n+2) \right) \left( \frac{6}{4n+3} \cdot \frac{1}{t} + \frac{12n+18}{4n+3} t - \frac{12(n+2)}{4n+3} \right).
    \end{split}
    \]
    And, 
    \[
    \begin{split}
       &\| \overset{\circ}{\mathrm{Ric}_{g_{t}}} \|^{2}_{g_{t}} -  R_{g_{t}} \left( \lambda_{Ric_{t}}^{V} - \frac{R_{g_{t}}}{\mathrm{dim}M} \right) = \frac{16n}{t}((2n+3)t - 1)(t-1)(3t - 2(n+2)), \\
       &\| \overset{\circ}{\mathrm{Ric}_{g_{t}}} \|^{2}_{g_{t}} -  R_{g_{t}} \left( \lambda_{Ric_{t}}^{H} - \frac{R_{g_{t}}}{\mathrm{dim}M} \right) \\
       &= \frac{12}{(4n+3)t^{2}} ((2n+3)t -1)(t-1)\left( (8n^{2} + 18n)t^{2} - 16n(n+2)t + 4n-3 \right).
    \end{split}
    \]
    Hence, if $t_{-} < t < \frac{1}{2n+3}$, $\frac{1}{2n+3} < t < 1$, $1 < t < t_{+}$ or $t > \frac{2}{3}(n+2)$,  then 
    \[
    \| \overset{\circ}{\mathrm{Ric}_{g_{t}}} \|_{g_{t}}^{2} \cdot g_{t} - R_{g_{t}} \left( \mathrm{Ric}_{g_{t}} - \frac{R_{g_{t}}}{\mathrm{dim}M} \cdot g_{t} \right)
    \]
    is positive or negative definite on $M$.
    Here, $t_{\pm}$ are the real roots of the equation: $(8n^{2} + 18n)t^{2} - 16n(n+2)t + 4n-3 = 0$.
    Note that $0 < t_{-} < \frac{1}{2n+3}$ and $1 < t_{+} < \frac{2}{3}(n+2)$.
    Therefore, $(\mathbb{S}^{4n+3}, g_{t})$ satisfies the assumption of Theorems \ref{theo-1}, \ref{theo-2}, \ref{theo-3} and \ref{theo-4} if $t_{-} < t < \frac{1}{2n+3}$, $\frac{1}{2n+3} < t < 1$, $1 < t < t_{+}$ or $t > \frac{2}{3}(n+2)$, and satisfies the assumptions of Theorems \ref{theo-1} and \ref{theo-2} if $0 < t \neq t_{\pm}, \frac{1}{2n+3}, 1, \frac{2(n+2)}{3}$.
    Moreover, $g_{t}$ is a Yamabe metric if $t \ge \frac{4n+5}{3}$.

    On the other hand, one can directly check that for any $t_0 > 2(n+2)/3$, $R_{g_t} \cdot g_t \ge R_{t_0} \cdot g_{t_0}$ on $T \mathbb{S}^{4n+3}$ for all $t \in [2(n+2)/3,  t_0]$.
    \item ({\cite[Section 5]{matsuzawa1983einstein}} or {\cite[9.83 Example 3]{besse2007einstein}})~Consider the fibration $\mathbb{C}P^{2n+1} \rightarrow \mathbb{H}P^{n}$ whose fibers are $S^{2}(4)$ with constant sectional curvature $4$.
    Then
     \[
    R^{F} = 8,~R^{B} = 16n(n+2),~R^{M} = (4n+4)(4n+2).
    \]
    Let $g_{t}$ be the canonical variation of this Riemannian submersion.
    Then, as in the previous example, one can calculate as
    \[
    \mathrm{Ric}^{t}(U,V) = \left( \frac{4}{t} + 4nt \right)\, g_{t}(U, V), ~
        \mathrm{Ric}^{t}(X,Y) = \left( 4(n+2) -4t  \right)\, g_{t}(X, Y)
    \]
    and 
    \[
    R_{g_{t}} = \frac{8}{t} -8nt +16n(n+2).
    \]
    Therefore, 
    \[
    \begin{split}
    \| \overset{\circ}{\mathrm{Ric}_{g_{t}}} \|^{2}_{g_{t}} &= 2 \left( \frac{4}{t} + 4nt -\frac{1}{4n+2}\left( \frac{8}{t} -8nt +16n(n+2) \right) \right)^{2} \\
    &+ 4n \left( 4(n+2) -4t -\frac{1}{4n+2}\left( \frac{8}{t} -8nt +16n(n+2) \right) \right)^{2}\\
    &= 2 \left( \frac{8n}{2n+1} \cdot \frac{1}{t} + \frac{2n(4n+4)}{2n+1} t - \frac{8n(n+2)}{2n+1} \right)^{2} \\
&+ 4n \left( \frac{4}{2n+1} \cdot \frac{1}{t} + \frac{4(n+1)}{2n+1} t - \frac{4(n+2)}{2n+1} \right)^{2}.
    \end{split}
    \]
    Let $\lambda^{V}_{Ric_{t}}$ and $\lambda^{H}_{Ric_{t}}$ be the eigenvalues of $\mathrm{Ric}_{g_{t}}$ in the vertical and horizontal directions, respectively.
    One can observe that 
    \[
    \begin{cases}
        \lambda_{Ric_{t}}^{V} \ge \lambda_{Ric_{t}}^{H} &0 < t \le \frac{1}{n+1}~\mathrm{or}~t \ge 1 \\
         \lambda_{Ric_{t}}^{H} \ge \lambda_{Ric_{t}}^{V} &\frac{1}{n+1} \le t \le 1.
    \end{cases}
    \]
    Matsuzawa \cite{matsuzawa1983einstein} observed that the canonical variation $g_{t}$ is an Einstein metric on $\mathbb{C}P^{2n+1}$ if and only if $t =1$ or $t = \frac{1}{n+1}$.
    Then, we have
    \[
    \begin{split}
        &R_{g_{t}} \left( \lambda_{Ric_{t}}^{V} - \frac{R_{g_{t}}}{\mathrm{dim}M} \right) \\
        &= \left( \frac{8}{t} -8nt + 16n(n+2) \right) \left( \frac{8n}{2n+1} \cdot \frac{1}{t} + \frac{2n(4n+4)}{2n+1} t - \frac{8n(n+2)}{2n+1} \right), \\
         &R_{g_{t}} \left( \lambda_{Ric_{t}}^{H} - \frac{R_{g_{t}}}{\mathrm{dim}M} \right) \\
         &= \left( \frac{8}{t} -8nt + 16n(n+2) \right) \left( \frac{4}{2n+1} \cdot \frac{1}{t} + \frac{4(n+1)}{2n+1} t - \frac{4(n+2)}{2n+1} \right).
    \end{split}
    \]
    And, 
    \[
\begin{split}
   &\| \overset{\circ}{\mathrm{Ric}_{g_{t}}} \|^{2}_{g_{t}} - R_{g_{t}} \left( \lambda_{Ric_{t}}^{V} - \frac{R_{g_{t}}}{\mathrm{dim}M} \right) 
   = \frac{64 n}{t^{2}} (t-n-2)((n+1)t - 1)(t-1), \\
   &\| \overset{\circ}{\mathrm{Ric}_{g_{t}}} \|^{2}_{g_{t}} - R_{g_{t}} \left( \lambda_{Ric_{t}}^{H} - \frac{R_{g_{t}}}{\mathrm{dim}M} \right) \\
   &\quad= \frac{32}{(2n+1)t^{2}} ((n+1)t -1)(t-1)\left( n(2n+3)t^{2} -4n(n+2)t +2n-1 \right).
\end{split}
    \]
    Hence, if $t_{-} < t < \frac{1}{n+1}$, $\frac{1}{n+1} < t < 1$, $1 < t < t_{+}$, or $t > n+2$, then 
    \[
    \| \overset{\circ}{\mathrm{Ric}_{g_{t}}} \|_{g_{t}}^{2} \cdot g_{t} - R_{g_{t}} \left( \mathrm{Ric}_{g_{t}} - \frac{R_{g_{t}}}{\mathrm{dim}M} \cdot g_{t} \right)
    \]
    is positive or negative definite on $M$.
Here, $t_{\pm}$ are the real roots of the equation: $n(2n+3)t^{2} -4n(n+2)t +2n-1 = 0$.
Note that $0 < t_{-} < \frac{1}{n+1}$ and $1 < t_{+} < n+2$.
Therefore, $(\mathbb{C}P^{2n+1}, g_{t})$ satisfies the assumption of Theorems \ref{theo-1}, \ref{theo-2}, \ref{theo-3} and \ref{theo-4}, and satisfies the assumptions of Theorems \ref{theo-1} and \ref{theo-2} if $0 < t \neq t_{\pm}, \frac{1}{n+1}, 1, n+2$. 
    Moreover, $g_{t}$ is a Yamabe metric if $t \ge 2n+3$.

    On the other hand, one can directly check that for any $t_0 > n+2$, $R_{g_t} \cdot g_t \ge R_{g_{t_0}} \cdot g_{t_0}$ on $T \mathbb{C}P^{2n+1}$ for all $t \in [n+2, t_0]$.

    \item ({\cite[9.84 Example 4]{besse2007einstein}})
    Let $\pi : \mathbb{S}^{15} \rightarrow \mathbb{S}^{8}$ be the Hopf fibratioin with fiber $\mathbb{S}^{7}$.
    Then, the metric $g$ on the total space $\mathbb{S}^{15}$ and $g^{F}$ on the fiber $\mathbb{S}^{7}$ have constant curvature $1$, and the metric $g^{B}$ on the base has constant curvature $4$. Hence, we have 
    \[
    R^{F} = R_{g_{\mathbb{S}^{7}(1)}} = 42, R^{M} = R_{g_{\mathbb{S}^{15}(1)}} = 210,
    R^{B} = R_{g_{\mathbb{S}^{8}(1/2)}} = 4 \cdot 56 = 224.
    \]
    And, let $g_{t}~(t > 0)$ be the canonical variation of $g$, then for vertical vectors $U, V$ and horizontal vectors $X, Y$, 
    \[
    \begin{split}
        \mathrm{Ric}_{g_{t}} (U, V) &= \left( \frac{6}{t} + 8t \right) \langle U, V \rangle_{t}, \\
        \mathrm{Ric}_{g_{t}} (X, Y) &= (28 - 14t) \langle X, Y \rangle_{t}.
    \end{split}
    \]
    Then, one can calculate as
    \[
    \begin{split}
    R_{g_{t}} &= \frac{42}{t} -56t + 224, \\
    \| \overset{\circ}{\mathrm{Ric}_{g_{t}}} \|^{2}_{g_{t}} &= \left( 7 \left( \frac{16}{15} \right)^{2} + 8 \left( \frac{14}{15} \right)^{2} \right) \frac{1}{t^{2}} (11t - 3)^{2} (t-1)^{2}, \\
    R_{g_{t}} \left( \lambda^{V}_{t} - \frac{R_{g_{t}}}{15} \right) &= \frac{14 \cdot 16}{15t^{2}} (3 + 6t -4t^{2})(11t-3)(t-1), \\
    R_{g_{t}} \left( \lambda^{H}_{t} - \frac{R_{g_{t}}}{15} \right) &= -\frac{14^{2}}{15 t^{2}} (3 + 6t -4t^{2})(11t-3)(t-1).
    \end{split}
    \]
    Moreover, 
    \[
    \begin{cases}
        \lambda^{V}_{t} \ge \lambda^{H}_{t} &t \le \frac{3}{11}~\mathrm{or}~t \ge 1 \\
        \lambda^{H}_{t} \ge \lambda^{V}_{t} &\frac{3}{11} \le t \le 1,
    \end{cases}
    \]
    and $(\mathbb{S}^{15}, g_{t})$ is Einstein if and only if $t = 1$ or $\frac{3}{11}$.
    On the other hand,
     \[
    \begin{split}
\| \overset{\circ}{\mathrm{Ric}_{g_{t}}} \|^{2}_{g_{t}} - R_{g_{t}} \left( \lambda^{V}_{t} - \frac{R_{g_{t}}}{15} \right) 
&= \frac{14 \cdot 16}{3} \frac{1}{t} \left( 3t-4 \right) (11t-3)(t-1), \\
\| \overset{\circ}{\mathrm{Ric}_{g_{t}}} \|^{2}_{g_{t}} - R_{g_{t}} \left( \lambda^{H}_{t} - \frac{R_{g_{t}}}{15} \right)
&= \frac{14 \cdot 2}{45} \frac{92}{t^2} \left( \left( t - \frac{49}{92} \right)^2 + \frac{7811}{92} \right) (11t - 3)(t-1).
\end{split}
    \]
    Therefore, $(\mathbb{S}^{15}, g_{t})$ satisfies the assumptions of Theorems \ref{theo-1} and \ref{theo-2} if $0 < t \neq \frac{3}{11}, 1, \frac{4}{3}$, and satisfies the assumptions of Theorems \ref{theo-1}, \ref{theo-2}, \ref{theo-3} and \ref{theo-4} if $t > 4/3$.
    Moreover, $g_{t}$ is a Yamabe metric if $t \ge 3$.

    On the other hand, one can directly check that for any $t_0 > 2$, $R_{t} \cdot g_t \ge R_{g_{t_0}} \cdot g_{t_0}$ on $T \mathbb{S}^{15}$ for all $t \in [2, t_0]$.
\end{itemize}

All of the manifolds in the above examples have positive Yamabe invariant.
\begin{ques}
    Are there any manifolds of non-positive Yamabe invariants that satisfies the assumptions of Theorems \ref{theo-1}, \ref{theo-2}, \ref{theo-3} and \ref{theo-4}?
\end{ques}
\subsection{Unstable csc metrics}\label{subsection-unstable}
\begin{itemize}
    \item (On $\mathbb{S}^{1} \times \mathbb{S}^{n-1}$)
Consider $(\mathbb{S}^{1} \times \mathbb{S}^{n-1}, g_{r} = r^{2} dt^{2} + g_{std})~(r > 0)$.
   Then, $\lambda_{1}(g_{r}) = \min \{ r^{-2}, n-1 \}$ and $R_{g_{r}} = (n-1)(n-2)$.
   So, $g_{r}$ is stable csc metric iff $r^{-1} \ge \sqrt{n-2}$. And we have
   \[
   (\overset{\circ}{\mathrm{Ric}}_{g_{r}})_{ij} =
   \begin{cases}
     -\frac{(n-1)(n-2)}{n} r^{2} dt^{2} & i=j=0 \\
     \frac{(n-2)}{n} (g_{std})_{ij} &1 \le i,j \le n-1 \\
     0&\mathrm{otherwise}.
   \end{cases}
   \]
   Here, $g_{std}$ denotes the standard spherical metric on $\mathbb{S}^{n-1}$ with constant curvature $1$.
   Hence, $\| \overset{\circ}{\mathrm{Ric}}_{g_{r}} \|^{2}_{g_{r}} = \frac{(n-1)(n-2)^{2}}{n}$,
   $\max \overset{\circ}{\mathrm{Ric}}_{g_{r}} = \frac{n-2}{n}$.
   Therefore, $\|\overset{\circ}{\mathrm{Ric}}_{g_{r}} \|^{2}_{g_{r}} + R_{g_{r}} \cdot \max \overset{\circ}{\mathrm{Ric}}_{g_{r}} = \frac{2(n-1)(n-2)^2}{n} > 0$ for all $r > 0$.
   Hence, $(\mathbb{S}^1 \times \mathbb{S}^{n-1}, g_r)~(r > 0)$ satisfies all the assumptions of Theorems  \ref{theo-1} and \ref{theo-2}.
   Moreover, from \cite{kobayashi1985large, schoen1987variational}, $g_r$ is the unique csc metric in its conformal class (up to rescaling) if $r^{-1} \ge \sqrt{n-2}$. Hence, by Theorem \ref{theo-cscdeformation}, $(\mathbb{S}^1 \times \mathbb{S}^{n-1}, g_r)~(0 < r \le 1/\sqrt{n-2})$ can be deformed in the csc direction so that neither Type I nor II rigidities hold.
   
   On the other hand, 
   \[
\| \overset{\circ}{\mathrm{Ric}}_{g_r} \|_{L^2}\, g_r - R_{g_r}\, \overset{\circ}{\mathrm{Ric}}_{g_{r}} > 0~~\mathrm{on}~M
   \]
   if and only if 
   \begin{equation}\label{eq-cscassumption}
   r > \frac{1}{2\pi\, \mathrm{Vol}(\mathbb{S}^{n-1}, g_{std})} = \frac{\Gamma \left( \frac{n-1}{2} \right)}{4 \pi^{\frac{n+1}{2}}}.
   \end{equation}
   As a result, if
   \[
    \frac{\Gamma \left( \frac{n-1}{2} \right)}{4 \pi^{\frac{n+1}{2}}} < r \le \frac{1}{\sqrt{n-2}},~n \ge 3
   \]
   then $(\mathbb{S}^1 \times \mathbb{S}^{n-1}, g_r)$ can be deformed in the csc direction so that Type III rigidity does not hold.
    \item (On $\mathbb{S}^{2} \times \mathbb{S}^{2}$)
    Consider $(\mathbb{S}^{2} \times \mathbb{S}^{2}, g_{r} = r^{2} g_{std} + g_{std})$.
    $\lambda_{1} (g_{r}) = \min \{ 2r^{-2}, 2 \}$, $\frac{R_{g_{r}}}{n-1} = \frac{2r^{-2} + 2}{3}$.
    So, $g_{r}$ is stable csc iff $1/\sqrt{2} \le r \le \sqrt{2}$.
    And we have
    \[
   (\overset{\circ}{\mathrm{Ric}}_{g})_{ij} =
   \begin{cases}
     \left( \frac{1}{2} - \frac{r^{2}}{2} \right) (g_{std})_{ij} & i,j = 1,2 \\
     \left( \frac{1}{2} - \frac{r^{-2}}{2} \right) (g_{std})_{ij} & i.j = 3,4 .
   \end{cases}
   \]
   Here, $g_{std}$ denotes the standard spherical metric on $\mathbb{S}^{2}$ with constant curvature $1$.
   Hence, $\| \overset{\circ}{\mathrm{Ric}}_{g_{r}} \|^{2}_{g_{r}} = (1-r^{-2})^{2}$, $R_{g_{r}} = 2r^{-2} + 2$ 
   and $\max \overset{\circ}{\mathrm{Ric}}_{g_{r}} = \max \left\{ \frac{r^{-2}}{2} - \frac{1}{2}, \frac{1}{2} - \frac{r^{-2}}{2} \right\}$.
   Therefore, 
   \[
   \| \overset{\circ}{\mathrm{Ric}}_{g_{r}} \|_{g_{r}}^{2} + R_{g_{r}} \cdot \max \overset{\circ}{\mathrm{Ric}}_{g_{r}} =
   2r^{-4}(1-r^2)
   \begin{cases}
       > 0 &(r < 1) \\
       < 0 &(r > 1).
   \end{cases}
   \]
\end{itemize}
Hence, $(\mathbb{S}^2 \times \mathbb{S}^2, g_r)~(r \neq 1)$ satisfies all the assumptions of Theorems \ref{theo-1} and \ref{theo-2}.

\noindent
The Einstein metric $g_1$ is not conformally equivalent to $(\mathbb{S}^4, [g_{std}])$, hence Obata's theorem implies that $g_1$ is the unique csc metric in its conformal class.
Thus, from \cite{bohm2004variational}, any csc metric sufficiently close to $g_1$ is also the unique csc metric in its conformal class.
Hence, $(\mathbb{S}^2 \times \mathbb{S}^2, g_r)$ where $r \neq 1$ is sufficiently close to $1$ can be deformed in the csc direction so that neither Type I nor II rigidities hold.
On the other hand, 
\[
\| \overset{\circ}{\mathrm{Ric}}_{g_r} \|_{L^2}\, g_r - R_{g_r}\, \overset{\circ}{\mathrm{Ric}}_{g_{r}} > 0~~\mathrm{on}~M
\]
if and only if
\[
r \ge 1,~r^{-4} + 16 \pi^2 r^{-3} - 32 \pi^2 r^{-1} + 16 \pi^2 r - 1 > 0,
\]
or 
\[
r \le 1,~-r^{-4} + 16 \pi^2 r^{-3} + 16 \pi^2 r + 1 > 0.
\]
One can easily check that $r = \sqrt{2}, \frac{1}{\sqrt{2}}$ satisfy these conditions, respectively.
Hence, these also hold for $r$ that is sufficiently close to those values.
Therefore, if $g_r$ where $r \ge \frac{1}{\sqrt{2}}$ (resp. $r \le \sqrt{2}$) is sufficiently close to $\frac{1}{\sqrt{2}}$ (resp. $\sqrt{2}$) is a unique Yamabe metric (up to rescaling) in its conformal class, then such a metric $g_r$ can be deformed in the csc direction so that Type III rigidity does not hold by Theorem \ref{theo-cscdeformation}.

\section{Conclusion}\label{section-conclusion}
\subsection{Other rigidity results}
B\"{a}r \cite{bear2024dirac} recently proved the following rigidity result.
\begin{theo}[{\cite[Main Theorem]{bear2024dirac}}]
    Let $(M, g)$ be a connected closed Riemannian spin manifold of dimension $ \ge 2$ and $D$ the Dirac operator acting on spinor fields of $M$.
    Then
    \begin{equation}\label{eq-hyperspherical}
    \lambda_{1}(D^{2}) \le \frac{n^{2}}{4\, \mathrm{Rad}_{\mathbb{S}^{n}} (M,d_{g})}.
    \end{equation}
    Equality holds in (\ref{eq-hyperspherical}) if and only if $(M, g)$ is isometric to $(\mathbb{S}^{n}(R), g_{std})$ with $R = \mathrm{Rad}_{\mathbb{S}^{n}}(M, d_{g})$.
\end{theo}
\noindent
Here, $\mathrm{Rad}_{\mathbb{S}^{n}}(M, d_{g})$ is the \textit{hyperspherical radius} of $M$, which is defined as the supremum of all numbers $R > 0$ such that there exists a Lipschitz map $f : (M, f_{g}) \rightarrow (\mathbb{S}^{n}, d_{g_{std}})$
with Lipschitz constant $\mathrm{Lip}(f) \le 1/R$ and $\mathrm{deg}(f) \neq 0$.

As mentioned in subsection 3.1 in \cite{bear2024dirac}, if $\min_{M} R_{g} > 0$, then
\[
\mathrm{Rad}_{\mathbb{S}^{n}}(M ,d_{g})^{2} \le \frac{n(n-1)}{\min_{M} R_{g}},
\]
and equality holds if and only if $(M, g)$ is isometric to $(\mathbb{S}^{n}, g_{std})$.
This especially implies Llarull's rigidity theorem (\cite{llarull1998sharp, cecchini2024lipschitz, lee2022rigidity})\footnote{See the remarks immediately following the Theorem 1 in \cite{bear2024dirac}.}.
Note that for any Lipschitz map $f : M \rightarrow \mathbb{S}^{n}$ with $\mathrm{deg}(f) \neq 0$, 
\[
\mathrm{Lip}(f) \ge \frac{1}{\mathrm{Rad}_{\mathbb{S}^{n}}(M, d_{g})}.
\]
Hence, for any Lipschitz map $f : M \rightarrow \mathbb{S}^{n}$ with $\mathrm{deg}(f) \neq 0$, 
\begin{equation}\label{eq-lip-scal}
    \frac{R_{g}}{\mathrm{Lip}(f)^{2}} \ge \frac{R_{g_{std}} \circ f}{\mathrm{Lip}(\mathrm{id}_{\mathbb{S}^{n}})^{2}} = n(n-1)~~\mathrm{on}~M
\end{equation}
implies $f : (M, d_{g}) \rightarrow (\mathbb{S}^{n}, g_{std})$ is an isometry.
Here, $\mathrm{id}_{\mathbb{S}^{n}} : \mathbb{S}^{n} \rightarrow \mathbb{S}^{n}$ denotes the identity map, hence $\mathrm{Lip}(\mathrm{id}_{\mathbb{S}^{n}}) = 1$.

\noindent
Motivated by this, we call a Riemanniam manifold $(M_{0}, g_{0})$ \textit{extremal in the sense of} (\ref{eq-lip-scal}) if 
\[
 \frac{R_{g}}{\mathrm{Lip}(f)^{2}} \ge R_{g_{0}} \circ f~~\mathrm{on}~M
\] 
for any Lipschitz map $f : M \rightarrow M_{0}$ with $\mathrm{deg}(f) \neq 0$ implies that $f : (M, d_{g}) \rightarrow (M_{0}, g_{0})$ is an isometry.
\begin{ques}
    What kinds of properties does extremal metric in the sense of (\ref{eq-lip-scal}) have?
    Can we find sufficient conditions for a metric not to be extremal in this sense, as in our main theorems?
\end{ques}

Similarly, from {\cite[Theorem 4]{bear2024dirac}}, for any Lipschitz map $f : M \rightarrow \mathbb{S}^{n}$ with $\mathrm{deg}(f) \neq 0$,
\begin{equation}\label{eq-yamabe-scal}
\frac{Y(M, [g])}{\mathrm{Lip}(f)^{2} \mathrm{Vol}(M, g)^{2/n}} \ge \frac{Y(\mathbb{S}^{n}, [g_{std}])}{\mathrm{Lip}(\mathrm{id}_{\mathbb{S}^{n}})^{2} \mathrm{Vol}(\mathbb{S}^{n}, g_{std})^{2/n}}
\end{equation}
implies that $(M, g)$ is isometric to $(\mathbb{S}^{n}, g_{std})$.
Here, $Y(M, [g])$ is the Yamabe constant of the conformal class $[g]$ on $M$ (see Remark \ref{rema-minscal} above).

\noindent
Motivated by this, we call a Riemanniam manifold $(M_{0}, g_{0})$ \textit{extremal in the sense of} (\ref{eq-yamabe-scal}) if 
\[
\frac{Y(M, [g])}{\mathrm{Lip}(f)^{2} \mathrm{Vol}(M, g)^{2/n}} \ge \frac{Y(M_{0}, [g_{0}])}{\mathrm{Vol}(M_{0}, g_{0})^{2/n}}
\] 
for any Lipschitz map $f : M \rightarrow M_{0}$ with $\mathrm{deg}(f) \neq 0$ implies that $(M, g)$ is isometric to $(M_{0}, g_{0})$.
\begin{ques}
    What kinds of properties does extremal metric in the sense of (\ref{eq-yamabe-scal}) have?
    Can we find sufficient conditions for a metric not to be extremal in this sense, as in our main theorems?
\end{ques}

\subsection{For singular metrics}\label{subsection-singular}
\begin{ques}
Let $M$ be a compact smooth manifold and $\Sigma \subset M$ is a closed subset.
Can we find a sufficient condition for a metric not to be scalar curvature rigid with a given ``boundary condition'' associated with $\Sigma$.
For example, 
\begin{itemize}
    \item $\Sigma = \partial M$ and the ``boundary condition'' is something involving the mean curvature of $\partial M$, or
    \item $\Sigma$ is an arbitrary closed subset and consider the set $\Sigma$ as the set of singular points of the metric
    in some sense.
    Here, the ``boundary condition'' is the decay of metric near the singular set $\Sigma$.
\end{itemize}

\end{ques}
Next, we mention that the scalar minimum functional $R_{min}$ (Section \ref{rema-minscal}) and a generalized definition of scalar curvature bounded below.
Let $M^{n}$ be a smooth closed $n$-manifold and $\kappa \in \mathbb{R}$ a constant.
Let $C^{0}_{met}(M, \kappa)$ denote the $C^{0}$-completion of $C^{2}$-metrics whose scalar curvature is bounded below by $\kappa$ in the conventional sense.
That is, a $C^{0}$-metric $g$ is in $C^{0}_{met}(M, \kappa)$ if and only if there exists a sequence of $C^{2}$-metrics $g_{i}$ on $M$ such that $g_{i}$ converges uniformly to $g$ and satisfies $R(g_{i}) \ge \kappa$. 
From the observation after Remark 1.10 of \cite{burkhardt2019pointwise} and Theorem 1.7 of the same paper, one can observe that
a $C^{0}$-metric $g$ is in $C^{0}_{met}(M, \kappa)$ if and only if 
\[
\limsup_{C^{2} \ni h \rightarrow g, C^{0}}  R_{min} (h) \ge \kappa.
\]
Here, ``$\limsup_{C^{2} \ni h \rightarrow g, C^{0}}  R_{min} (h) \ge \kappa$'' means that 
\[
\lim_{\delta \rightarrow 0} \left( \sup_{h \in \mathcal{M}^{2},~\| h -g \|_{C^{0}(M, g)} < \delta} R_{min}(h) \right) \ge \kappa,
\]
where $\mathcal{M}^{2}$ denotes the set of all $C^{2}$-metrics on $M$. 
Note that if 
\[\limsup_{C^{2} \ni h \rightarrow g, C^{0}}  R_{min} (h) := \alpha < + \infty,
\]
then 
this is equivalent to the property that for any $\varepsilon > 0$ there is $\delta > 0$ such that
\[
R_{min}(h) < \alpha + \varepsilon
\]
for all $h \in \mathcal{M}^{2}$ with $\| h - g \|_{C^{0}(M, g)} < \delta$
and there exists a sequence of $C^{2}$-metrics $h_{i}$ on $M$ such that $R_{min}(h_{i}) \rightarrow \alpha$ as $i \rightarrow \infty$.

Consider the following particular situation.
Let $M$ be a closed manifold with non-positive Yamabe invariant $Y(M) \le 0$.
Then, as mentioned in Remark \ref{rema-minscal}, 
\[
Y(M) = \sup_{h \in \mathcal{M}} R_{min}(h) \cdot \mathrm{Vol}(M, h)^{2/n}.
\]
Assume that there is a sequence of $C^{2}$-metrics $g_{i}$ on $M$ such that $g_{i} \overset{C^{0}}{\longrightarrow} g \in C^{0}$ and
     \begin{equation}\label{eq-nonpositive-yamabe}
         R_{min} (g_{i}) \cdot \mathrm{Vol}(M, g_{i})^{2/n} \rightarrow Y(M)\, (\le 0)~~\mathrm{as}~i \rightarrow \infty.
     \end{equation}
Then, $R_{min}(g_{i}) \rightarrow Y(M) \cdot \mathrm{Vol}(M, g)^{-2/n}$ as $i \rightarrow \infty$ and hence
\[
\limsup_{C^{2} \ni h \rightarrow g, C^{0}}  R_{min} (h) \ge \lim_{i \rightarrow \infty} R_{min}(g_{i}) = Y(M) \cdot \mathrm{Vol}(M, g)^{-2/n}.
\]
Therefore, 
$g \in C^{0}_{met}\left( M, Y(M) \cdot \mathrm{Vol}(M, g)^{-2/n} \right)$.
Of course, a typical example of $(g_{i})$ that satisfies only (\ref{eq-nonpositive-yamabe}) is a sequence of solutions to the Yamabe problem on each conformal class $[g_{i}]$.
In relation to this observation,
it is interesting to explore some relations between the variational properties of $R_{min}$ and singular Yamabe metrics or other extremal metrics in a topology weaker than $C^{2}$.

\bigskip
\noindent
As mentioned just after Theorem \ref{theo-main-unstable}, it is interesting to study Listing-types of rigidity theorems for the following types of metrics.

$M^n$ is a closed $n$-manifold, $K \subset M$ is a closed subset and it has a decomposition
\[
K = K^{n-1} \sqcup K^{n-2},
\]
where $K^{n-1}$ is a smooth $(n-1)$-submanifold of $M$ and $K^{n-2}$ is a closed countably $(n-2)$-rectifiable subset of $M$. 
And $g$ is a smooth Riemannian metric in $M \setminus K$.
(As a weaker version of this, we might consider the case where the metric is further Lipschitz continuous on the whole manifold $M$.)
Under this setting, the problem is:
Can such a metric $g$ be deformed in the direction that is transverse to the conformal direction?
More precisely, is there a family of metrics $\{ g_s \}_{|s| < \varepsilon}$ such that $g_0 = g$, each $g_s$ is smooth,  $\mathrm{tr}_g \frac{\partial}{\partial s}g_s = 0$ in $M \setminus K$, and
\[
R_{g_s} > R_g \cdot \| g \|_{1, g_{s}}~\mathrm{or}~R_{g_{s}} > R_g \cdot \| g \|_{2, g_{s}}~~\mathrm{in}~M?
\]

\smallskip
\noindent
Although several generalizations of the Llarull-types of rigidity theorem are known for several types of singular metrics \cite{chu2024llarull, lee2022rigidity}, (to the best of the author's knowledge,) Listing-type of rigidity for singular metrics is not yet known.

\subsection{About our assumptions}
\begin{ques}
    In our theorems \ref{theo-1}, \ref{theo-2}, \ref{theo-3} and \ref{theo-4}, can we weaken the assumption that ``$\| \overset{\circ}{\mathrm{Ric}_{g}} \|_{g}(x) \neq 0$ for all $x \in M$'' to that ``the metric $g$ is not an Einstein metric''?
\end{ques}
\begin{ques}
    Does every closed manifold $M$ of dimension $n \ge 3$ admit a metric $g$ with non-positive constant scalar curvature and $\| \overset{\circ}{\mathrm{Ric}_{g}} \|_{g}(x) \neq 0$ for all $x \in M$?
\end{ques}
\noindent
Several examples of manifolds that admit no Einstein metric are known.
On the other hand, every closed manifold of dimension $\ge 3$ admits a metric with constant negative scalar curvature.
Therefore, such a manifold of dimension $\ge 3$ always admits a metric of non-Einstein negative constant scalar curvature.
Moreover, Matsuo {\cite[Corollary 2]{matsuo2014prescribed}} proved that there exists a non-Ricci-flat scalar-flat metric on every closed manifold of dimension $\ge 3$ with positive Yamabe invariant.
However, one cannot distinguish whether the norm of the traceless Ricci tensor of each such metric has a positive lower bound on the whole manifold.

\smallskip
According to Theorem \ref{theo-3}, we propose the following problem.
\begin{prob}
    For a metric $g$ whose scalar curvature is constant, study the image of the map
    \[
    \mathcal{TT}_{g} \ni h \mapsto \langle \mathrm{Ric}_{g}, h \rangle_{g}\, g - \langle \mathrm{Ric}_{g}, g \rangle_{g}\, h \in \mathcal{S}^{2}TM
    \]
    where $\mathcal{TT}_{g}$ denotes the space of all $(0,2)$-tensors that are traceless and divergence-free with respect to $g$.
    In particular, investigate when it becomes positive or negative definite.
\end{prob}


\section{Appendix}\label{section-appendix}
Let $g$ and $\bar{g}$ be two Riemannian metrics on a $n$-manifold.
Set the difference between the Levi-Civita connections of $g$ and $\bar{g}$ as
\[
W := \nabla - \bar{\nabla}.
\]
Then $W$ is a tensor (unlike $\Gamma$).
With respect to a local frame $e_{1}, \cdots, e_{n}$, we can write the components of $W$ via
\[
(\nabla_{i} - \bar{\nabla}_{i})(e_{j}) = W^{k}_{ij} e_{k}.
\]
First, direct computations deduce the following two propositions.
\begin{prop}\label{prop-W}
In a local coordinates,
    \[
    W^{k}_{ij} = \frac{1}{2} g^{kl} (\bar{\nabla}_{i} g_{lj} + \bar{\nabla}_{j} g_{il} - \bar{\nabla}_{l} g_{ij}).
    \]
    Here, $\bar{\nabla}_{i} g_{jk}$ denotes the expression of $\nabla^{\bar{g}} g$ in terms of the local coordinates.
\end{prop}
\begin{proof}
    \[
\begin{split}
    \bar{\nabla}_{i} g_{lj} &= \partial_{i} g_{lj} - g_{pj} \bar{\Gamma}^{p}_{il} - g_{lp} \bar{\Gamma}^{p}_{ij} \\
    \bar{\nabla}_{j} g_{il} &= \partial_{j} g_{il} - g_{pl} \bar{\Gamma}^{p}_{ji} - g_{ip} \bar{\Gamma}^{p}_{jl} \\
    -\bar{\nabla}_{l} g_{ij} &= -\partial_{l} g_{ij} + g_{pj} \bar{\Gamma}^{p}_{il} + g_{ip} \bar{\Gamma}^{p}_{lj}.
    \end{split}
    \]
    Taking the sum of both sides, we get
    \[
g^{kl} (\bar{\nabla}_{i} g_{lj} + \bar{\nabla}_{j} g_{il} - \bar{\nabla}_{l} g_{ij}) = 2 \Gamma^{k}_{ij} - \delta^{k}_{p} \bar{\Gamma}^{p}_{ij} -\delta^{k}_{p} \bar{\Gamma}^{p}_{ji}
= 2 \Gamma^{k}_{ij} - 2 \bar{\Gamma}^{k}_{ij}
= 2 W^{k}_{ij}.
    \]
\end{proof}
\begin{prop}
\label{prop-scal-differential}
In a local coordinates,
    \begin{equation}\label{eq-ricci-formula}
    R_{ij} = \bar{R}_{ij} + \bar{\nabla}_{k} W^{k}_{ij} - \bar{\nabla}_{i} W^{k}_{kj} + W^{p}_{ij} W^{k}_{kp} - W^{p}_{kj} W^{k}_{ip}.
    \end{equation}
    Here, $R_{ij}$ and $\bar{R}_{ij}$ denote the expressions of the Ricci tensors of $g$ and $\bar{g}$ respectively, in terms of the local coordinates.
\end{prop}

\begin{proof}
    \[
    \begin{split}
        W^{p}_{ij} W^{k}_{kp} &= (\Gamma^{p}_{ij} - \bar{\Gamma}^{p}_{ij}) (\Gamma^{k}_{kp} - \bar{\Gamma}^{k}_{kp}) = \Gamma^{p}_{ij} \Gamma^{k}_{kp} - \underset{(1)}{\underline{\bar{\Gamma}^{p}_{ij} \Gamma^{k}_{kp}}} - \underset{(2)}{\underline{\Gamma^{p}_{ij} \bar{\Gamma}^{k}_{kp}}} + \underset{(3)}{\underline{\bar{\Gamma}^{p}_{ij} \bar{\Gamma}^{k}_{kp}}} \\
        -W^{p}_{kj} W^{k}_{ip} &= -(\Gamma^{p}_{kj} - \bar{\Gamma}^{p}_{kj}) (\Gamma^{k}_{ip} - \bar{\Gamma}^{k}_{ip}) = -\Gamma^{p}_{kj} \Gamma^{k}_{ip} + \underset{(4)}{\underline{\bar{\Gamma}^{p}_{kj} \Gamma^{k}_{ip}}} + \underset{(6)}{\underline{\Gamma^{p}_{kj} \bar{\Gamma}^{k}_{ip}}} - \underset{(5)}{\underline{\bar{\Gamma}^{p}_{kj} \bar{\Gamma}^{k}_{ip}}} \\
        \bar{\nabla}_{k} W^{k}_{ij} &= \bar{\nabla}_{k} (\Gamma^{k}_{ij} - \bar{\Gamma}^{k}_{ij}) \\
        &= \partial_{k} \Gamma^{k}_{ij} - \partial_{k} \bar{\Gamma}^{k}_{ij} + \underset{(2)}{\underline{\Gamma^{p}_{ij} \bar{\Gamma}^{k}_{kp}}} - \underset{(7)}{\underline{\Gamma^{k}_{pj} \bar{\Gamma}^{p}_{ik} }}- \underset{(4)}{\underline{\Gamma^{k}_{ip} \bar{\Gamma}^{p}_{kj}}}
        - \bar{\Gamma}^{p}_{ij} \bar{\Gamma}^{k}_{kp} + \underset{(8)}{\underline{\bar{\Gamma}^{k}_{pj} \bar{\Gamma}^{p}_{ik}}} + \underset{(5)}{\underline{\bar{\Gamma}^{k}_{ip} \bar{\Gamma}^{p}_{kj}}} \\
         -\bar{\nabla}_{i} W^{k}_{kj} &= -\bar{\nabla}_{i} (\Gamma^{k}_{kj} - \bar{\Gamma}^{k}_{kj}) \\
        &= -\partial_{i} \Gamma^{k}_{kj} + \partial_{i} \bar{\Gamma}^{k}_{kj} - \underset{(6)}{\underline{\Gamma^{p}_{kj} \bar{\Gamma}^{k}_{ip}}} + \underset{(7)}{\underline{\Gamma^{k}_{pj} \bar{\Gamma}^{p}_{ki}}} + \underset{(1)}{\underline{\Gamma^{k}_{kp} \bar{\Gamma}^{p}_{ij}}}
        + \bar{\Gamma}^{p}_{kj} \bar{\Gamma}^{k}_{ip} - \underset{(8)}{\underline{\bar{\Gamma}^{k}_{pj} \bar{\Gamma}^{p}_{ki}}} - \underset{(3)}{\underline{\bar{\Gamma}^{k}_{kp} \bar{\Gamma}^{p}_{ij}}}.
    \end{split}
    \]
    Therefore, we have (where terms with the same number cancel each other)
    \[
\begin{split}
    \bar{\nabla}_{k} W^{k}_{ij} &- \bar{\nabla}_{i} W^{k}_{kj} + W^{p}_{ij} W^{k}_{kp} - W^{p}_{kj} W^{k}_{ip} \\
    &= \partial_{k} \Gamma^{k}_{ij} - \partial_{i} \Gamma^{k}_{kj} + \Gamma^{p}_{ij} \Gamma^{k}_{kp} - \Gamma^{p}_{kj} \Gamma^{k}_{ip} \\
    &~~~~~-(\partial_{k} \bar{\Gamma}^{k}_{ij} - \partial_{i} \bar{\Gamma}^{k}_{kj} + \bar{\Gamma}^{p}_{ij} \bar{\Gamma}^{k}_{kp} - \bar{\Gamma}^{p}_{kj} \bar{\Gamma}^{k}_{ip}) \\
    &= R_{ij} - \bar{R}_{ij}.
\end{split}
\]
\end{proof}
\begin{prop}\label{prop-scal-first-derivative}
    $DR|_{\bar{g}} (h) = -\Delta_{\bar{g}} (tr_{\bar{g}} h) + \mathrm{div}_{\bar{g}} (\mathrm{div}_{\bar{g}} h) - \langle \mathrm{Ric}_{\bar{g}}, h \rangle_{\bar{g}}.$
\end{prop}
\begin{proof}
    From Proposition \ref{prop-scal-differential} ($g = g_{t} := \bar{g} + th~(|t| << 1)$), we have
    \[
    R_{g_{t}} = (\bar{g}^{ij} - th^{ij}) \bar{R}_{ij} + (\bar{g}^{ij} - th^{ij})(\bar{\nabla}_{k} W^{k}_{ij} - \bar{\nabla}_{j} W^{k}_{ki}) + \mathrm{other~terms}.
    \]
    Now, since $W|_{t = 0} = 0$, the ``other terms'' is vanishing when $t = 0$.
    From Proposition \ref{prop-W}, we have
    \[
    W^{k}_{ij} = \frac{1}{2} (\bar{g}^{kl} - t\bar{h}^{kl}) \left( \bar{\nabla}_{i}(\bar{g}_{lj} + th_{lj}) + \bar{\nabla}_{j} (\bar{g}_{il} + th_{il}) - \bar{\nabla}_{l}(\bar{g}_{ij} + th_{ij}) \right).
    \]
    Since $\bar{\nabla} \bar{g} = 0$, we have
    \[
    \frac{d}{dt} W^{k}_{ij}|_{t = 0} = \frac{1}{2} \bar{g}^{kl} (\bar{\nabla}_{i} h_{lj} + \bar{\nabla}_{j} h_{il} - \bar{\nabla}_{l} h_{ij}).
    \]
    Summing these up, we have
    \[
    \begin{split}
        \frac{d}{dt} R_{g_{t}}|_{t = 0} &= -\langle \mathrm{Ric}_{\bar{g}}, h \rangle + \bar{g}^{ij} \left( \bar{\nabla}_{k} \frac{d}{dt} W^{k}_{ij}|_{t=0} - \frac{d}{dt} \bar{\nabla}_{j} W^{k}_{ki}|_{t=0} \right) \\
        &= -\langle \mathrm{Ric}_{\bar{g}}, h \rangle + \frac{1}{2} \bar{g}^{ij} \left( \bar{\nabla}_{k} \bar{g}^{kl} (\bar{\nabla}_{i}h_{lj} + \bar{\nabla}_{j} h_{il} -\bar{\nabla}_{l}h_{ij}) \right) \\
        &- \frac{1}{2} \bar{g}^{ij} \left( \bar{\nabla}_{j} \bar{g}^{kl} (\bar{\nabla}_{k} h_{li} + \bar{\nabla}_{i} h_{kl} - \bar{\nabla}_{l} h_{ki}) \right) \\
        &= -\langle \mathrm{Ric}_{\bar{g}}, h \rangle + \frac{1}{2} \bar{g}^{ij} \bar{g}^{kl} \bar{\nabla}_{k} (\underset{(1)}{\underline{\bar{\nabla}_{i}h_{lj}}} + \bar{\nabla}_{j} h_{il} - \underset{(2)}{\underline{\bar{\nabla}_{l}h_{ij}}}) \\
        &- \frac{1}{2} \bar{g}^{ij} \bar{g}^{kl} \bar{\nabla}_{j} (\bar{\nabla}_{k} h_{li} + \underset{(2)}{\underline{\bar{\nabla}_{i} h_{kl}}} - \underset{(1)}{\underline{\bar{\nabla}_{l} h_{ki}}}) \\
        &=  -\langle \mathrm{Ric}_{\bar{g}}, h \rangle -\underset{(2)}{\underline{\bar{g}^{kl} \bar{\nabla}_{k} \bar{\nabla}_{l} (\bar{g}^{ij} h_{ij})}} + \underset{(1)}{\underline{\bar{g}^{ij} \bar{g}^{kl} \bar{\nabla}_{k} \bar{\nabla}_{i} h_{lj}}} \\
        &= -\langle \mathrm{Ric}_{\bar{g}}, h \rangle - \bar{\Delta} (\mathrm{tr}_{\bar{g}} h) + \mathrm{div}_{\bar{g}} (\mathrm{div}_{\bar{g}} h).
    \end{split}
    \]
    (The terms with the same number have canceled each other.)

    \noindent
    Here we have used 
    \begin{itemize}
        \item $W|_{t=0} = 0$ in the 1st equality,
        \item $\bar{\nabla} \bar{g} = 0$ in the 3rd and 4th equality,
        \item In the 4th equality, the other term is vanishing.
    \end{itemize}
\end{proof}
\noindent
Moreover, a more detailed calculation shows that if $g = \bar{g} + h~(\| h \|_{\bar{g}} << 1)$,
\[
R_{g} = \bar{R} + DR|_{\bar{g}}(h) + (\bar{g}+h)^{-1} h \bar{g}^{-1} h  \bar{g}^{-1} * \mathrm{Ric}_{\bar{g}} + g^{-1} * g^{-1} * g^{-1} * \bar{\nabla} h * \bar{\nabla} h,
\]
where the term $g^{-1} * g^{-1} * g^{-1} * \bar{\nabla} h * \bar{\nabla} h$ is a contraction of three copies of $g^{-1}$ (i.e., $g$ with raised indices) and two copies of $\bar{\nabla} h = \bar{\nabla} g$.
And, the term $(\bar{g}+h)^{-1} h \bar{g}^{-1} h \bar{g}^{-1} * \mathrm{Ric}_{\bar{g}}$ is the trace of $\mathrm{Ric}_{\bar{g}}$ with respect to $\left( (\bar{g}+h)^{-1} h \bar{g}^{-1} h \bar{g}^{-1} \right)^{-1}$.
Note that $\bar{g} + h$ is positive definite if $\| h\|_{\bar{g}}$ is small enough.
Indeed, this formula follows by taking both sides of (\ref{eq-ricci-formula}) in Proposition \ref{prop-scal-differential} with respect to $g = \bar{g} + h$ and 
using 
\[
(\bar{g} + h)^{-1} =\bar{g}^{-1} -\bar{g}^{-1} h \bar{g}^{-1} + (\bar{g}+h)^{-1} h \bar{g}^{-1} h \bar{g}^{-1}.
\]
The term $g^{-1} * g^{-1} * g^{-1} * \bar{\nabla} h * \bar{\nabla} h$ comes from the ``other terms'' in the proof of Proposition \ref{prop-scal-first-derivative}.

\section*{Acknowledgements}

The author was supported by JSPS KAKENHI Grant Number 24KJ0153.

\bigskip
\noindent
\textit{E-mail adress}:~hamanaka1311558@gmail.com

\smallskip
\noindent
\textsc{Department of Mathematics, Graduate School of Pure and Applied Sciences, University of Tsukuba, Tsukuba, Ibaraki, Japan}

\end{document}